\documentclass[12pt]{amsart}

\usepackage{amsmath, amsthm, amssymb, amscd, multirow, xcolor, placeins, booktabs, array}

\newif\ifpdfAuthoring
\pdfAuthoringtrue
\ifpdfAuthoring
\usepackage[
bookmarks=true,
bookmarksnumbered=true,
bookmarksopen=false,
hypertexnames=false,
breaklinks=false,
colorlinks=true,
pagebackref=true,
hyperindex=true,
plainpages=false,
pdfauthor={Hatice Boylan and Nils-Peter Skoruppa},
pdftitle={Jacobi Forms of Lattice Index I. Basic Theory},
pdfsubject={2022},
pdfkeywords={article},
]{hyperref}
\usepackage{ae}
\else
\newcommand {\texorpdfstring}[2] {#1}
\fi

\newcommand {\C}         {{\mathbb C}}
\newcommand {\Eis}[1]    {\mathcal{E}_{#1}}

\newcommand {\Gdual}[1]  {{#1}^\ast}
\newcommand {\HP}        {{\mathbb H}}
\newcommand {\Mp}[1][\Z] {\sym{Mp}(2,{#1})}
\newcommand {\Q}         {{\mathbb Q}}

\newcommand {\SL}[1][\Z] {\sym{SL}(2,{#1})}
\newcommand {\Z}         {{\mathbb Z}}
\newcommand {\bif}       {\beta}
\newcommand {\disc}[1]   {D_{#1}}
\newcommand {\distjac}[1]{\vartheta_{#1}}
\newcommand {\dual}[1]   {{#1}^\sharp}
\newcommand {\ev}[1]     {{#1}_{\rm ev}}

\newcommand {\isom}      {\cong}
\newcommand {\lat}[1][L] {\underline #1}
\newcommand {\leg}[2]    {\genfrac{(}{)}{}{}{#1}{#2}}
\newcommand {\lsum}      {\perp}
\newcommand {\mat}[4]    {\left[\begin{smallmatrix} #1&#2 \\#3&#4\end{smallmatrix}\right]}

\newcommand {\pry}       {p}
\newcommand {\sectheta}  {\vartheta^*}
\newcommand {\sh}[1]     {{#1}^{\bullet}}
\newcommand {\sym}[1]    {\operatorname{#1}}

\newcommand {\tr}        {\sym{tr}}

\newcommand {\ve}        {\varepsilon}

\usepackage[breakable]{tcolorbox}

\theoremstyle{plain}
\newtheorem{Theorem}{Theorem}[section]
\newtheorem{Lemma}[Theorem]{Lemma}
\newtheorem{Proposition}[Theorem]{Proposition}

\newtheorem*{Main Lemma}{Main Lemma}
\newtheorem*{Supplement-1}{Supplement to the Main Lemma}
\newtheorem*{Theorem-NoNum}{Theorem}
\newtheorem*{Supplement}{Supplement}

\theoremstyle{definition}
\newtheorem{Definition}[Theorem]{Definition}
\newtheorem*{Remark}{Remark}

\newtheorem{Example}[Theorem]{Example}

\begin{document}


\title[Jacobi Forms of Lattice Index]{%
  Jacobi Forms of Lattice Index I.\\
  Basic Theory}

\author{%
  Hat\.ice Boylan
}
\address{%
  {\.I}stanbul {\"U}niversitesi, Fen Fak{\"u}ltesi, Matematik B{\"o}l{\"u}m{\"u},
  34134 Vez\-ne\-ci\-ler, Fatih/{\.I}stanbul, Turkey
}
\email{%
  hatice.boylan@gmail.com}

\author{%
  Nils-Peter Skoruppa}
\address{%
  Department Mathematik, Universit\"at Siegen, 57068 Siegen, Germany}
\email{%
  nils.skoruppa@gmail.com}



\begin{abstract}
  This is the first one of a series of articles in which we develop
  the theory of Jacobi forms of lattice index, their close interplay
  with the arithmetic theory of lattices and the theory of Weil
  representations. We hope to publish this series in an extended and
  combined way eventually as a monograph.

  In this part we present the basic theory and first structure
  theorems.  We deduce explicit dimension formulas and give
  non-trivial explicit examples.
\end{abstract}

\maketitle
\tableofcontents


\section*{Introduction}
\label{sec:intro}

In this article we present various theorems describing the spaces
$J_{k,F}(\chi)$ of Jacobi forms of half integral or integral weight
$k$, of matrix index $F$, and with character~$\chi$ on the full
modular group as they were introduced
in~\cite{Skoruppa-Schiermonnikoog} (see also \cite{Arakawa},
\cite{Clery-Gritsenko} for one of the first and one of the recent
articles on this subject).  In Section~\ref{sec:lattice-index} we
recall the definition of these spaces. However, for this we shall use
a language that differs from the one used in the cited articles. More
specifically, instead of dealing with Jacobi forms of matrix index we
shall propose the essentially equivalent but more flexible notion of
{\em Jacobi forms of lattice index}.

For a symmetric non-degenerate
$n\times n$-matrix~$F$, the space $J_{k,F}(\chi)$ is trivial unless
$k\ge \frac n2$ and $2F$ has integer entries and is positive
definite~\cite{Skoruppa:Chern-Classes}. The weights $k=\frac n2$ and
$\frac {n+1}2$ are called the {\em singular} and {\em critical weight} for
Jacobi forms with matrix index of size~$n$, respectively.  The main
motivation for this article comes from the observation that for scalar
index the Jacobi forms of singular weight can be described in terms of
merely two universal functions, namely the functions
\begin{equation}
  \label{eq:simplest}
  \vartheta(\tau,z)=\sum_{r\in\Z} \leg{-4}r q^{\frac{r^2}8}\zeta^{\frac r2},
  \quad
  \sectheta(\tau,z)=\sum_{r\in\Z} \leg{12}r q^{\frac{r^2}{24}}\zeta^{\frac r2}
  .
\end{equation}
where $q=e^{2\pi i \tau}$, $\zeta=e^{2\pi i z}$. The function
$\vartheta(\tau,z)$ occurs in the famous Jacobi triple product
identity~\cite[p.~90]{Jacobi}
\begin{equation}
  \label{eq:vartheta-product-expansion}
  \vartheta(\tau,z)
  =
  q^{\frac 18}(\zeta^{\frac12}-\zeta^{-\frac12})\prod_{n\ge1}(1-q^n)(1-q^n\zeta)(1-q^n\zeta^{-1})
  .
\end{equation}
The function $\sectheta(\tau,z)$, can be written as
\begin{equation}
  \label{eq:sectheta-eta}
  \sectheta(\tau,z) = \frac{\vartheta(\tau,2z)}{\vartheta(\tau,z)}\eta(\tau),
  \qquad
  \eta(\tau)
  =
  q^{\frac1{24}}\prod_{n\ge1}(1-q^n)
  ,
\end{equation}
where $\eta(\tau)$ is the Dedekind eta function, and hence posesses
also a product expansion, which is known also as the Watson quintuple
product identity. These functions define Jacobi forms in
$J_{\frac12,\frac12}(\ve^3)$ and $J_{\frac12,\frac32}(\ve)$,
respectively~\cite[p.~27]{Skoruppa:Dissertation}. Here $\ve$ denotes
the linear character of $\Mp$ (the non-trivial two-fold central
extension of~$\SL$) afforded by the Dedekind eta function. The
character $\ve$ has order~$24$ and generates the group of linear
characters of $\Mp$ (cf.~Proposition~\ref{prop:G-abelianization}),
and~$\ve^2$ factors through a character of the modular group $\SL$ and
generates the group of linear characters of $\SL$.

The mentioned observation for the spaces of Jacobi forms of singular
weight and scalar index can be summarized as follows.
\begin{Theorem-NoNum}[\protect{~\cite{Skoruppa:Dissertation}}]
  \label{thm:singular-rank-1}
  Let $m$ be a positive half-integer. Then the space of Jacobi forms
  $J_{\frac12,m}(\chi)$ is trivial unless $(\chi, m)=(\ve,
  \frac{3d^2}2)$ or $(\chi, m)=(\ve^3, \frac{d^2}2)$ for a positive
  integer~$d$. In the latter cases one has
  \begin{equation*}
    J_{\frac12,\frac{3d^2}2}(\ve)=\C\cdot\sectheta(\tau,dz),
    \qquad
    J_{\frac12,\frac{d^2}2}(\ve^3)=\C\cdot\vartheta(\tau,dz)
    .
  \end{equation*}
\end{Theorem-NoNum}

It seems that this theorem holds also for Jacobi forms of lattice
index of arbitrary rank.  As explained above we shall develop the
basic features of this language in
Section~\ref{sec:lattice-index}. Using lattice indices the theorem can
be divided into two parts. The first part is that $\vartheta$
and~$\sectheta$ are non-zero Jacobi forms of singular weight for the
maximal positive definite integral lattices
$\lat [\Z]=\big(\Z,(x,y)\mapsto xy\big)$ and
$\lat [\Z](3)=\big(\Z,(x,y)\mapsto 3xy\big)$, respectively, and that
these are (up to isomorphism) the only maximal positive definite
integral lattices of rank one that permit non-zero Jacobi forms of
singular weight. The second part is that, for an arbitrary positive
definite integral lattice~$\lat$ of rank one and an arbitrary linear
character $\chi$ of~$\SL$, the space $J_{1/2,\lat}(\chi)$ equals
$J_{1/2,\lat^{\mathrm{m}}}(\chi)$, where $\lat^{\mathrm{m}}$ is the
maximal positive definite integral lattice containing $\lat$.

In a sequel to the present article we shall show that the first
statement extends to indices of arbitrary rank in the sense that, for
a given rank~$n$, there are only finitely many maximal positive
definite integral lattices of rank~$n$ which permit non-zero Jacobi
forms of singular weight.  The second statement holds modulo slight
modifications also for rank~two and in certain cases also for higher
rank. Currently we do not know whether or in what forms the second
statement holds for arbitrary rank $n>2$.

The key for proving these results is the connection between Jacobi
forms and vector-valued modular forms which, for an integral lattice
$\lat$ of rank $n$ as index, can be expressed as a natural isomorphism
\begin{equation}
  \label{eq:fundamental-isomorphism}
  J_{k,\lat}(\chi)
  \isom
  \big(M_{k-n/2}\otimes\Theta(\lat)\otimes\C(\chi)\big)^{G}
  ,
\end{equation}
where $G=\Mp$ is the non-trivial double cover of $\SL$, where
$M_{k-n/2}$ is the (infinite dimensional) $G$-module of modular forms
of weight $k-n/2$ on congruence subgroups of $G$, where $\C(\chi)$ is
$\C$ with the $G$-action $(\alpha,z)\mapsto \chi(\alpha)z$, and where
$\Theta(\lat)$ is the $G$-module spanned by the basic theta functions
associated to~$\lat$ (see~\eqref{eq:basic-theta-functions}). For
singular weight, i.e.~for $k=n/2$ and thus $M_{k-n/2}=\C$ we are thus
led to study the $G$-invariants of~$\Theta(\lat)\otimes\C(\chi)$.

For the latter it is important to observe that $\Theta(\lat)$ is
isomorphic as $G$-module to the Weil representation associated to the
discriminant module $\disc {\lat}$ of $\lat$, which exists, strictly
speaking, only for even $\lat$. It is not difficult to
see~\cite{Skoruppa-fqm} that the Weil representation depends (up to
$G$-module isomorphism) only on the isomorphism class of $\disc
{\lat}$ (as object in the category of finite quadratic modules). It is
therefore clear that results on singular weight Jacobi forms, and to a
certain extent also on higher weight, have their counterpart in the
theory of finite quadratic modules and associated Weil
representations.

As it turned out in our studies and in particular, in view of the
indicated results on Jacobi forms of singular weight, it would be
unnatural and unnecessary complicated to exclude odd lattices
(i.e.~integral but not even lattices) from our
considerations. However, for an odd lattice, the notion of
discriminant module and associated Weil representation does not
exist. We overcome this dilemma by associating to an odd lattice two
new objects, which we call its {\em shadow} and {\em shadow
  representation}. The shadow of a lattice is in fact not a new
notion. It was introduced in~\cite{Sloane-Rains}, where we also
borrowed its name from. As it turns out all what is sketched in the
previous paragraph holds true for odd lattices if one includes in the
theory the shadows and the shadow representations of odd lattices.  We
point out that the shadow representations might also be of interest in
the theory of Weil representations independent of its applications to
Jacobi forms.

The isomorphism~\eqref{eq:fundamental-isomorphism} has two further
important consequences.  The first one concerns Jacobi forms of
critical weight, i.e.~of weight $(n+1)/2$ and lattice index of rank
$n$. For this we recall a theorem of Serre and
Stark~\cite[Thm.~A]{Serre-Stark} which implies that $M_{1/2}$ equals
the direct limit of the family of $G$-modules $\Theta^{\mathrm{ev}}(\lat
[\Z](m))$, where $m$ runs through the positive integers, where the
superscript `ev' denotes the subspace of functions $\theta(\tau,z)$ in
$\Theta^{\mathrm{ev}}(\lat [\Z](m))$ which are even in $z$, and where,
for $m=nd^2$, the $G$-module $\Theta^{\mathrm{ev}}(\lat [\Z](n))$ is
mapped onto $\Theta^{\mathrm{ev}}(\lat [\Z](m))$ via
$\theta(\tau,z)\mapsto \theta(\tau,dz)$. It follows that
\begin{equation*}
  J_{{(n+1)}/2,\,\lat}(\chi)
  \isom
  \varinjlim_m
  J_{{(n+1)}/2,\,\lat [\Z](m)\lsum \lat}(\chi)
  .
\end{equation*}
In other words all Jacobi forms of critical weight and index of
rank~$n$ `come' from Jacobi forms of singular weight and index of rank
$n+1$.  A more explicit statement is given in
Theorem~\ref{thm:critical-iso}.

The second important consequence of~\eqref{eq:fundamental-isomorphism}
follows from the fact that $\Theta(\lat)$ and $\Theta(\lat')$ are
isomorphic if $\disc{\lat}$ and $\disc{\lat'}$ are isomorphic as
discriminant modules. If the latter holds true
then~\eqref{eq:fundamental-isomorphism} implies that
\begin{equation*}
  J_{k,\lat}(\chi)
  \isom
  J_{k',\lat'}(\chi)
  ,
\end{equation*}
for all $k$ and $k'$ such that $k-n/2=k'-n'/2$, where $n$ and $n'$ are
the ranks of $\lat$ and $\lat'$, respectively. We note that, in
particular, a dimension formula for~$J_{k,\lat}(\chi)$ should depend
only on $k-n/2$ and $\disc{\lat}$, but not on $\lat$. In
Section~\ref{sec:dimension-formulas} we work out an explicit dimension
formula for the spaces $J_{k,\lat}(\chi)$ and we shall recover the
mentioned phenomenon.

The underlying theme of this article (and of the planned subsequent
ones) is the close interplay between the theory of Jacobi forms and
the arithmetic theory of lattices.  We believe that this principle can
serve as a valuable guideline for producing still more insight into
the nature of Jacobi forms of lattice index than explained in this
first article.

When preparing our results on Jacobi forms of singular and critical
weight we observed that treatment of Jacobi forms as it can be found in
the literature was not appropriate to pursue our studies.  For
comparing and relating Jacobi forms of different index we needed a
coordinate free manner to deal with the indices. The necessary formal
steps for this led to an intrinsic formulation (and re-formulation)
which turned out to be more than a syntactic trick, but revealed more
and more how much purely lattice theoretic questions have their
counterpart in the theory of Jacobi forms. As a consequence our
material split naturally into two parts. One which develops
systematically the theory of Jacobi forms of coordinate free lattice
index, and another one which concentrates on properties of lattices
and their associated Weil representations or shadow representations.

In this article we develop in Section~\ref{sec:lattice-index} the
notion of Jacobi forms of lattice index. In
section~\ref{sec:theta-expansion} we explain the relation between
Jacobi forms and Weil- and shadow representations. However, these
representations are realized here as spaces of theta functions. A
deeper study of these representations profit from a purely algebraic
formulations, which will be developed in a sequel of this article. In
Section~\ref{sec:theta-expansion} we shall also explain various
important consequences of the mentioned relation between Jacobi forms
and Weil-and shadow representations. In
Section~\ref{sec:dimension-formulas} we deduce a ready to compute
dimension formula for spaces of Jacobi forms of lattice index. In
Section~\ref{sec:structure-theorems} we discuss the structure of the
graded space generated by all Jacobi forms of a given index and
character as module over the ring of modular forms of level~$1$, and
we give additional information for graded spaces of weakly holomorphic
Jacobi forms.  In the final Section~\ref{sec:examples} we apply the
results and ideas of the preceding sections to describe explicitly the
spaces of Jacobi forms on various lattices.

\subsection*{Notations}
We shall use throughout $e(x)=e^{2\pi i x}$ and $\HP$ for the upper
complex half plane.  For a ring $R$ the symbol $R^n$ denotes the
$R$-module of {\em row} vectors. Morever, we use $T=\mat 11{}1$ and
$S=\mat {}{-1}1{}$.


\section{Jacobi forms of lattice index}
\label{sec:lattice-index}

In this section we develop the theory of Jacobi forms of lattice
index.  This theory is equivalent to the theory of Jacobi forms of
matrix index~$F$ for symmetric positive definite $n\times n$ matrices
with entries in $\frac12\Z$. In this sense we propose here rather a
change of language than a new theory. However, a new ingredient will
be that we shall be able to treat odd lattices as index at the same
level as even indices by making use of the notion of a shadow of a
lattice. The inclusion of odd lattices in the theory will become
important in later sections when we study critical and singular
weight.

We introduce some notations. We use
\begin{equation*}
  G := \Mp
\end{equation*}
for the metaplectic cover of $\SL$. Recall that this is a central
extension
\begin{equation*}
  1\rightarrow \{\pm 1\}
  \rightarrow G
  \rightarrow \SL
  \rightarrow 1
  .
\end{equation*}
More precisely, $G$ consists of all pairs $(A,w)$ of elements $A=\mat
abcd$ of $\SL$ and holomorphic square roots $w(\tau)$ of the function
$c\tau+d$ on the upper half plane~$\HP$. The multiplication is defined
by $\big(A,w\big)\big(B,v\big)=\big(AB,w(A\tau)v(\tau)\big)$. Using
the presentation~$[s,t;s^2=(st)^3,s^4=1]$ of $\SL$ (see
e.g.~\cite[Thm.~1.2.5]{Rankin}) it is not difficult to show that $G$
is, up to isomorphism, the only non-trivial two-fold central extension
of $\SL$\footnote{Suppose $H$ is a non-trivial two-fold central
  extension of $\SL$. Choose $\sigma$ and~$\tau$ so that
  $\sigma\mapsto S$ and $\tau\mapsto T$ under the canonical projection
  from $H$ onto $\SL$ so that they satisfy
  $\sigma^2=(\sigma\tau)^3$. Clearly, $H$ is generated by $\sigma$ and
  $\tau$. Then either we have $\sigma^4=1$ or $\sigma^8=1$. However,
  the former can not hold true since $H$ is a non-trivial
  extension. This proves the claim.}. Recall that the group of linear
characters of $\SL$ is cyclic of order~$12$
\cite[Thm~1.3.1]{Rankin}. Recall also that
\begin{equation}
  \label{eq:epsilon}
  \ve\big(A,w\big)
  :=
  \frac{\eta(A\tau)}{\eta(\tau)\,w(\tau)}
  \qquad
  ((A,w)\in G)
  ,
\end{equation}
where $\eta(\tau)$ denotes the Dedekind eta function
(see~\eqref{eq:sectheta-eta}), is a constant depending only on
$A$~\cite{Dedekind1}, \cite{Dedekind2}.  It is clear that $\ve$
defines then a linear character of $G$.
\begin{Proposition}
  \label{prop:G-abelianization}
  The group of linear characters of the metaplectic cover $G$ of $\SL$
  is cyclic of order $24$, generated by the character $\ve$ defined
  in~\eqref{eq:epsilon}.
\end{Proposition}
\begin{proof}
  Since under the natural projection of $G$ onto $\SL$ the commutator
  subgroup $K$ of $G$ is mapped onto the commutator subgroup of $\SL$
  it follows that $K$ has index $12$ or $24$ in $G$.  Since
  $\ve(1,-1)=-1$ we see that $\ve^2$ factors through a linear
  character of $\SL$, and since $\ve^2(T)=e_{12}(1)$ and the
  abelianized group of $\SL$ has order $12$ we see that $\epsilon^2$
  has order $12$, whence $\ve$ has order $24$.  This implies the
  proposition.
\end{proof}

An integral lattice of rank $n$ is a pair $\lat = (L,\bif)$ of a free
$\Z$-module~$L$ of rank~$n$ and a $\Z$-valued symmetric $\Z$-bilinear
form $\bif$ on $L$. We shall use throughout
\begin{equation*}
  \bif(x):=\tfrac 12\bif(x,x)
  .
\end{equation*}
We shall use the same letter $\bif$ for the $R$-bilinear extension of
$\beta$ to $R\otimes_\Z L$, where~$R$ is any ring containing $\Z$.
Mostly, the lattices under consideration will be {\em positive
  definite}, i.e.~satisfy $\bif(x)>0$ for all $x\not=0$ in
$L$. However, at a few occasions and in particular, in the sequel of
this article we shall also consider not necessarily positive
lattices. We call $\lat$ {\em semi-positive definite}, if $\bif(x)\geq
0$ for all $x$ in $L$.

The simplest example for an integral lattice is the {\em standard
  lattice}
\begin{equation*}
  \lat[\Z]^n := \big(\Z^n,(x,y)\mapsto xy^t\big)
  .
\end{equation*}
We call an integral lattice $\lat$ {\em even} if $\bif(x,x)$ is even
for all $x$ in $L$, otherwise we call it {\em odd}. The standard
lattice $\lat[\Z]^n$ is odd. The map
\begin{equation}
  L\rightarrow \Z/2\Z,\quad
  x\mapsto e(\bif(x))
\end{equation}
defines a group homomorphism whose kernel we denote by $\ev L$. If
$\lat$ is even, the homomorphism is trivial and $\ev L =L$. Otherwise
$\ev L$ is a submodule of $L$ of index~$2$. Obviously,
\begin{equation*}
  \ev{\lat}:=(\ev L,\bif)
\end{equation*}
is the {\em maximal even sublattice of $\lat$.} Here and in the
following we call a lattice $\lat'$ {\em sublattice} of $\lat$ if
$\lat'=(L',\bif|_{L'})$ for a submodule $L'$ of finite index in $L$,
and we shall usually simply write $\bif$ for $\bif|_{L'}$. If $L$ is
non-degenerate (i.e.~one has $\bif(x,L)=0$ only for $x=0$), then there
exists a vector $r$ in $\Q\otimes L$ such that $\bif(r,x)\equiv
\bif(x)\bmod \Z$ for all $x$ in $L$. Following common
terminology~\cite{Sloane-Rains}, we then call the set
\begin{equation*}
  \sh L
  :=\big\{
  r\in \Q\otimes L:
  \bif(r,x)\equiv \bif(x)\bmod \Z\text{ for all }x\in L
  \big\}
\end{equation*}
the {\em shadow of $\lat$} and its elements the {\em shadow vectors of
  $\lat$}.  For a $\Z$-module $M$ of full rank $n$ in $\Q\otimes L$ we
use
\begin{equation*}
  \dual M = \{x \in \Q \otimes L : \bif(x,M) \subseteq \Z\}
\end{equation*}
for the {\em dual of $M$}.  Clearly $\sh L$ is a subset of $\dual{(\ev
  L)}$. If $\lat$ is even, then $\sh L=\dual L$. Otherwise $\sh L$ is
a coset, namely the non-trivial element in $\dual{\big(\ev
  L\big)}/\dual L$.  In particular,~$\sh L$ is invariant under
translation with elements in $L$, and we can use for the orbit (or
collection of cosets) $\sh L/L$. We have $|\sh L/L| = \det(\lat)$,
where
\begin{equation*}
  \det(\lat):=|\dual L/L|
  ,
\end{equation*}
or, equivalently, $\det(\lat)=\det(G)$ for any Gram matrix of $\lat$.
We define the {\em level} of an integral non-degenerate lattice
$\lat=(L,\bif)$ as the smallest positive integer $\ell$ such that
$\ell\bif(x)\in\Z$ for all $x$ in~$\dual {\ev L}$. Note that this is
the traditional definition for even lattices, and that the level of an
integral lattice $\lat=(L,\bif)$ coincides with the level of $(\ev
L,\bif)$.

\begin{Example}
  For the standard lattice $\lat[\Z]^n$ the lattice $\ev\Z^n$ consists
  of all integral vectors $x=(x_1,\dots,x_n)$ such that
  $x_1+\cdots+x_n$ is even, the shadow equals the collection of
  vectors $\frac 12 x$ in $\frac 12\Z^n$ such that all components
  of~$x$ are odd, and the level of $\lat[\Z]^n$ equals~$2$, $4$ or~$8$
  according as $n$ is $0$, $2$ or~$1$ modulo~$4$, respectively. The
  lattice $\ev{(\lat [\Z]^n)}$ equals the root lattice $\lat [D]_n$.
\end{Example}

If $8$ divides $n$ then $x \cdot x$ is integral for all~$x$ in $\sh
{\Z^n}$. In particular, the level of a general lattice $\lat
=(L,\beta)$ does in general not equal the smallest positive
integer~$\ell$ such that $\ell\bif(x)$ takes integer values on~$\sh
L$.

We can now give the definition of Jacobi forms with lattice index.
\begin{Definition}
  \label{def:Jacobi-forms-of-lattice-index}
  For $k$ in $\frac12\Z$, a positive definite integral lattice
  $\lat=(L,\bif)$, and an integer $h$ the space $J_{k,\lat}(\ve^h)$ of
  {\em Jacobi forms of weight~$k$, index $\lat$ and character~$\ve^h$}
  consists of all holomorphic functions $\phi(\tau,z)$ of a variable
  $\tau\in\HP$ and a variable $z\in \C\otimes L$ which satisfy the
  following properties:
  \begin{enumerate}
  \item
    \label{it:G-invariance}
    For all $\alpha=\big(A=\mat abcd,w\big)$ in $G$ one has
    \begin{equation*}
      \phi\big(\tfrac{a\tau+b}{c\tau+d},\tfrac z{c\tau+d}\big)
      e\big(\tfrac{-c\,\bif(z)}{c\tau+d}\big)\,
      w(\tau)^{-2k}
      =
      \ve^h(\alpha)\,\phi(\tau,z) .
    \end{equation*}
  \item
    \label{it:L-invariance}
    For all $x,y \in L$ one has
    \begin{equation*}
      \phi(\tau,z+x\tau + y)\,
      e\big( \tau \bif(x) + \bif(x,z)\big)
      =
      e\big( \bif(x+y)\big)
      \,\phi(\tau,z) .
    \end{equation*}
  \item
    \label{it:expansion}
    The function $\phi$ is {\em holomorphic at infinity}, i.e.~its
    Fourier development is of the form
    \begin{equation*}
      \phi(\tau,z)
      =
      \sum_{n\in \frac h{24} +\Z}
      \sum_{\begin{subarray}{c} r\in \sh L\\
          n \ge \bif(r)
        \end{subarray}} c(n,r) \, q^n \, e\big(\bif(z,r)\big) .
    \end{equation*}
  \end{enumerate}
\end{Definition}

Note that the crucial point in~\eqref{it:expansion} is the condition
$n\ge \bif(r)$.  That $\phi$ has Fourier expansion with $n$ and $r$
running through $\frac 1{24}h +\Z$ and~$\sh L$ holds true for any
holomorphic $\phi(\tau,z)$ satisfying the transformation
laws~\eqref{it:G-invariance} and~\eqref{it:L-invariance}, as one
easily sees by applying these transformation laws to $\tau\mapsto
\tau+1$ and to $z\mapsto z+1$. Sometimes it is useful to consider
holomorphic function $\phi(\tau,z)$ which
satisfy~\eqref{it:G-invariance},~\eqref{it:L-invariance} and
\eqref{it:expansion} with the condition $n\ge \bif(r)$ replaced by
$n\ge N$ for some integer $N$. Following the language used in the
theory of scalar-valued Jacobi forms we call such forms {\em weakly
  holomorphic} and denote the space of such forms
by~$J_{k,\lat}^{!}(\ve^h)$.

Note also that the transformation law~\eqref{it:G-invariance}, applied
to $\alpha=(1,-1)$, implies that $J_{k,\lat}(\ve^h)=0$ unless
$\ve^h(1,-1)=(-1)^{2k}$. In other words, we have
\begin{Proposition}
  One has $J_{k,\lat}(\ve^h)=0$ unless $k\equiv h/2 \bmod \Z$.
\end{Proposition}
In particular, if $k$ is integral then $J_{k,\lat}(\ve^h)=0$ unless
$h$ is even, so that $\ve^h$ factors through a character of $\SL$. In
this case the transformation law~\eqref{it:G-invariance} is in fact a
transformation law under $\SL$. If $k \in \frac 12+\Z$ then
$J_{k,\lat}(\ve^h)=0$ unless~$h$ is odd, i.e.~unless $\ve^h$ is a {\em
  genuine} character of $\Mp$.

The left hand sides of~\eqref{it:G-invariance}
and~\eqref{it:L-invariance} define (right-)actions of $G$ and $L\times
L$ on the space of holomorphic functions $\phi$ on
$\HP\times\big(\C\otimes L\big)$, both of which we denote for the
moment by a slash. It is a trivial but tedious exercise to show that,
for any $\alpha=(\mat abcd,w)$ in $G$ and any $(x,y)$ in $L\times L$,
one has
\begin{equation}
  \label{eq:semi-direct-product-obstruction}
  \phi|(x,y)|\alpha
  =
  e\big(ab\bif(x)+cd\bif(y)\big)\,\phi|\alpha|(x',y')
  ,
\end{equation}
where $(x',y')=(x,y)\mat abcd=(ax+cy,bx+dy)$. This identity has two
consequences. Namely, if a $\phi\not=0$
satisfies~\eqref{it:G-invariance} and~\eqref{it:L-invariance}, but
with the factor $e\big( \bif(x+y)\big)$ on the right
of~\eqref{it:L-invariance} replaced by the value $\chi(x,y)$ of a
given linear character of $L\times L$,
then~\eqref{eq:semi-direct-product-obstruction} implies
\begin{equation}
  \label{eq:chi-obstruction}
  \chi(x,y)
  =\chi(ax+cy,bx+dy)\,e\big(ab\bif(x)+cd\bif(y)\big)
\end{equation}
for all $x$, $y$ in $L$ and $\mat abcd$ in $\SL$.  We leave it to the
reader to verify the following
\begin{Proposition}
  For a non-degenerate integral lattice $\lat=(L,\beta)$, the character
  $\chi(x,y)=e\big(\bif(x+y))$ is the only character of $L\times L$
  which satisfies~\eqref{eq:chi-obstruction} for all $\mat abcd$ in
  $\SL$.
\end{Proposition}

Secondly, we see from~\eqref{eq:semi-direct-product-obstruction} that
the transformations laws~\eqref{it:G-invariance}
and~\eqref{it:L-invariance} can be combined to yield an action of the
semi-direct product $J(\lat):=G\ltimes (L\times L)$, where the
semi-direct product is defined with respect to the natural right
action of $\SL$ (and hence $\Mp$) on pairs $(x,y)$ in $L\times
L$. However, the result $\phi|(x,y)$ of the action of a pair $(x,y)$
on $\phi$ has to be redefined as the left hand side
of~\eqref{it:L-invariance} multiplied by $e\big(-\bif(x+y))$. The
transformation laws~\eqref{it:G-invariance}
and~\eqref{it:L-invariance} become then equivalent to the statement
that $\phi|g=\phi$ for all $g$ in $J(\lat)$.

We defined Jacobi forms only with respect to integral and positive
definite lattices. One might wonder whether it could make sense to
admit also non-integral or not positive definite lattices.  Suppose
$\lat = (L,\beta)$ is an arbitrary lattice. The left hand side
of~\eqref{it:L-invariance} defines then a projective action of
$L\times L$ on functions on~$\HP\times(\C\otimes L)$, namely
$\phi|(x,y)|(x',y')=e\big(-\bif(x,y')\big)\,\phi|(x+x',y+y')$. If
there would be a function $\phi\not=0$ and a map $\chi$ on $L\times
L\rightarrow \C^*$ satisfying $\phi|h=\chi(h)\phi$ for all $h$ in
$L\times L$, then $\chi$ would split the cocycle
$e\big(-\bif(x,y')\big)$, i.e.~one would have
$e\big(-\bif(x,y')\big)=\chi(x+x',y+y')/\chi(x,y)\chi(x',y')$ for all
$x,y,x',y'$ in $L$. Such a function $\chi$ obviously exist only if
$\bif$ takes integral values. This is because $\chi$, being defined on an abelian
group, is symmetric.

Finally, suppose $\lat = (L,\bif)$ is integral (but not necessarily
positive definite).
\begin{Proposition}
  If there is a non-zero holomorphic function
  $\phi:\HP\times(\C\otimes L)\rightarrow \C$, which satisfies
  $\phi|(x,y)=e\big(\bif(x+y)\big)\phi$ ($x,y\in L$), then $\lat$ is
  semi-positive definite.
\end{Proposition}
\begin{proof}
  Fix $\tau$ in $\HP$. Set $F(z)=\phi(z)H(z)\overline{\phi(z)}$, where
  $H(z)=e\big(-4\pi \bif(\Im z)/\Im \tau\big)$. It is easy to see that
  $F(z+x\tau+y)=F(z)$ for all $x$, $y$ in $L$. Morever, $F\ge
  0$. Therefore, $F$ can be viewed as a smooth function on
  $T:=\C\otimes L\big/\tau\otimes L+L$. Note that $T$ is a complex torus,
  hence it is compact.  Therefore $F$ assumes a maximum value, say at
  $z_0$. Hence, $-\log(F)$ assumes its minumum at $z_0$ and it is
  smooth in a neighborhood of $z_0$. Therefore the Hessian of
  $-\log(F)$ at $z_0$ is semi-positive definite. Let $\mat A{*}{**}D$
  be the Hessian of $-\log(F)$ at $z_0$. Then $A, D\geq 0$.

  Let $L=\Z a_1+\ldots \Z a_n$ and let $z_j$ be the coordinate
  functions with respect to $a_j$.
  We have\footnote{$C$ is referred to as the Chern class of the line bundle
      $L\times L\backslash\big((\C\otimes L)\times \C\big) \rightarrow T$,
      where the action of $L\times L$ on
      $(\C\otimes L)\times \C$ is defined by
      \begin{equation*}
        \big((x,y), (z,\xi))\mapsto
        (z+\tau x+y, \xi e\big(-\tau \bif(x)-\bif(x,z)-\bif(x+y)\big)
        .
      \end{equation*}} 
    \begin{multline*}
      C
      :=
      -\big(\tfrac{\partial}{\partial z_h}\tfrac{\partial}{\partial \overline{z_l}}\log F\big)_{h,l}
      =
      \\
      -\left(\tfrac{\partial}{\partial z_h}\tfrac{\partial}{\partial \overline{z_l}}
        \big(\log \phi+\log \overline{\phi}-4\pi\bif(\Im z)/\Im \tau\big)\right)_{h,l}
      =
      c\, G
      ,
  \end{multline*}
  where $c$ is a positive constant and $G$ is the Gram matrix of
  $\lat$ with respect to $a_j$. The last identity follows from the
  fact that $\tfrac{\partial}{\partial \overline{z_l}} \log \phi=0$
  and $\tfrac{\partial}{\partial z_h}\log \overline{\phi}=0$ for any
  $h$ and $l$ (see Cachy-Riemann equations). On the other hand, using
  $\tfrac{\partial}{\partial z_h}=\tfrac{\partial}{\partial
    x_h}-i\tfrac{\partial}{\partial y_h}$ and
  $\tfrac{\partial}{\partial \overline{z_l}}=\tfrac{\partial}{\partial
    x_l}+i\tfrac{\partial}{\partial y_l}$, we obtain that $C=c'(A+D)$,
  where $c'$ is a positive constant. Therefore, we have that $G\geq
  0$, which proves the proposition. 
 \end{proof}

Thus, the only freedom which remains in the definition of Jacobi forms
is that the inteegral lattice $\lat$ can be taken as semi-positive
definite.  As we shall see in a moment the case of semi-positive
definite lattices can be reduced to the case of positive definite
ones, so that we shall not need to include the former ones in our
theory.

It remains to answer the natural question whether, for a given
integral and positive definite lattice and character $\ve^h$, there
exist always non-zero Jacobi forms of, say, sufficiently big weight in
$k\in h/2+\Z$.

For this we note some obvious but useful formal properties of our
definition. If~$\phi$ and $\psi$ are Jacobi forms in
$J_{k,\lat}(\ve^h)$ and $J_{k',\lat'}(\ve^{h'})$ for indices
$\lat=(L,\bif)$ and $\lat'=(L',\bif')$, respectively, then we can form
the product $\phi\otimes \psi$, by which we mean the function
\begin{equation*}
  \{\phi\otimes \psi\}(\tau,z_1\oplus z_2)
  =
  \phi(\tau,z_1)\psi(\tau,z_2),
\end{equation*}
defined on $\HP\times [(\C\otimes L)\oplus (\C\otimes L')]$.  It is
clear that $\phi\otimes \psi$ defines a Jacobi form in
$J_{k+k',\,\lat\lsum\lat'}(\ve^{h+h'})$, where the sum $\lat \lsum
\lat'$ is defined in the obvious way.  If $L=L'$, we can also define
the product $\phi\psi$ in the usual way, namely
$\{\phi\psi\}(\tau,z)=\phi(\tau,z)\psi(\tau,z)$. This is then a Jacobi
form in $J_{k+k',\,(L,\bif+\bif')}(\ve^{h+h'})$.

If $\lat'$ is a sublattice of $\lat$, then $J_{k,\lat'}(\chi)$
contains $J_{k,\lat}(\chi)$ as a subspace.
More generally, any isometric map $\alpha:\lat'\rightarrow\lat$
defines a linear map
\begin{equation}
  \label{eq:alpha-star}
  \alpha^*:
  J_{k,\lat}(\chi) \rightarrow J_{k,\lat'}(\chi),
  \quad
  \big(\alpha^*\phi)(\tau,z)
  =
  \phi\big(\tau,\alpha (z)\big)
  .
\end{equation}
In particular, the map $(\phi,\alpha)\mapsto \alpha^*\phi$ defines a
right action of the orthogonal group~$O(\lat)$ of $\lat$ on each space
$J_{k,\lat}(\ve^h)$. Note also that, for Jacobi forms $\phi$ and
$\psi$ as in the preceding paragraph and in the case $L=L'$, one has
$\phi\psi=\alpha^*(\phi\otimes\psi)$, where $\alpha$ is the embedding
$z\mapsto z\oplus z$ of $(L,\bif+\bif')$ into $\lat \lsum \lat'$.

The simplest non-trivial Jacobi forms are those functions defined
in~\eqref{eq:simplest}.
\begin{Lemma}[\protect{\cite[p.~27]{Skoruppa:Dissertation}}]
  \label{lem:theta}
  The functions $\vartheta(\tau,z)$ and $\sectheta(\tau,z)$
  in~\eqref{eq:simplest} define Jacobi forms in
  $J_{\frac12,\lat[\Z]}\big(\ve^3\big)$ and
  $J_{\frac12,\lat[\Z](3)}\big(\ve\big)$, respectively (where $z$
  denotes the coordinate function on $\C\otimes \Z$ which maps
  $w\otimes 1$ to $w$).
\end{Lemma}
Here, for a lattice $\lat=(L,\bif)$ and for a positive integer $a$ we
use $\lat(a)$ for the lattice $(L,a\bif)$. Another proof for the
modularity of $\vartheta(\tau,z)$ can also be found in~\cite[App.~I,
Thm 5.4]{Skoruppa:Manifolds-and-Modular-Forms}.

From the lemma we obtain immediately Jacobi forms for the standard
lattice~$\lat [\Z]^N$ in the form
\begin{equation}
  \label{eq:singular-form-for-Zn}
  \distjac {\lat [\Z]^N}(\tau,\_)
  :=
  \vartheta(\tau,w_1)\cdots \vartheta(\tau,w_n)
  ,
\end{equation}
where the $w_j$ are the coodinate functions on $\C\otimes \Z^N$ with respect to the canonical $\Z$-basis of $\Z^N$ (formed by row vectors
$(1,0,\dots,0)$, $(0,1,0,\dots,0)$, \dots). A similar
construction can be applied to $\sectheta$ for obtaining Jacobi forms
for $\lat [\Z]^N(3)$. Using the maps~\eqref{eq:alpha-star} we see,
that any isometric embedding of a lattice $\lat$ into $\lat [\Z]^N$ or
$\lat [\Z]^N(3)$ yields then immediately a Jacobi form for
$\lat$. More precisely, we have
\begin{Proposition}
  \label{prop:pullback}
  For any isometric embedding $\alpha:\lat\rightarrow\lat[\Z]^N$, one
  has
  \begin{equation}
    \label{eq:thetablock}
    \alpha^*\distjac {\lat [\Z]_N}(\tau,z)
    =
    \prod_{j=1}^N \vartheta\big(\tau,\alpha_j(z)\big)
    \in
    J_{N/2,\lat}\big(\ve^{3N}\big),
  \end{equation}
  and for any isometric embedding
  $\alpha:\lat\rightarrow\lat[\Z]^N(3)$, one has
  \begin{equation}
    \label{eq:secthetablock}
    \prod_{j=1}^N \sectheta\big(\tau,\alpha_j(z)\big)
    \in
    J_{N/2,\lat}\big(\ve^{N}\big)
    .
  \end{equation}
  Here the $\alpha_j$ denote the coordinate functions of $\alpha$.
\end{Proposition}

\begin{Theorem}
  For any integral positive definite lattice $\lat$ and any integer
  $h$, there exist $k\gg 0$ and $k\equiv h/2\bmod \Z$ such that
  $J_{k,\lat}(\ve^h)\not=0$.
\end{Theorem}
\begin{proof}
  If $\lat$ admits an isometric embedding into the standard lattice
  $\lat [\Z]^N$ for some $N$ then the assertion of the theorem follows
  from the construction~\eqref{eq:thetablock} (by stripping off the
  components of $\alpha$ which are identically zero, and by
  multiplying the left hand side of ~\eqref{eq:thetablock} by a
  suitable non-negative power of the $\eta$-function if necessary, for
  obtaining the given character~$\ve^h$).

  However, not every integral lattice can be embedded into $\lat
  [\Z]^N$ for some~$N$. The first rank where this occurs is rank $6$,
  where the lattice~$E_6$ does not admit such an
  embedding~\cite[Thm.~1]{Conway-Sloane-V}.

  If a lattice $\lat$ does not admit embeddings in $\lat [\Z]^N$ the
  theorem follows from the dimension formula in
  Theorem~\ref{thm:dimension-formula} below, according to which we
  have
  \begin{equation}
    \label{eq:dimensions-asymptotics}
    \dim J_{k,\lat}(\ve^h)
    =
    \tfrac k{24}
    (
    \det(\lat) +
    (-1)^{k-h/2+n_2}\,2^{n-n_2}
    )
    +O(1)
  \end{equation}
  for $k\to\infty$, $k\equiv h/2\bmod \Z$. Here $n_2$ is the rank of
  the unimodular Jordan constituent of $\lat$ over $\Z_2$ (see
  Lemma~\ref{lem:weak-Jordan-decomposition}).

  Note that the last argument actually shows that
  $J_{k,\lat}(\ve^h)\not=0$ for all sufficiently large~$k\equiv
  h/2\bmod\Z$.

  Alternatively, one could also apply the usual methods and use
  suitably defined Eisenstein series or, more generally, Poincar\'e
  series, to prove the existence of Jacobi forms for given weight,
  lattice and character. However, we shall not pursue such constructions
  in this article.
\end{proof}

In the definition of Jacobi forms we assumed from the beginning on
that the index is positive definite. However, as the previous
discussion indicated there might also be a theory for semi-positive
definite lattices. For such a theory one would have to modify suitably
the property~\eqref{it:expansion} in
Definition~\ref{def:Jacobi-forms-of-lattice-index} since $\sh L$ as
defined here, is, for a degenerate lattice, no longer a lattice. It is
not difficult to find the right definition for holomorphicity at
infinity (e.g.~by using the isomorphism of
Proposition~\ref{prop:semi-definite-index} below). However, as we
shall see in a moment the consideration of semi-positive definite
indices can easily be reduced to positive definite ones.
  
For this let $\lat=(L,\bif)$ be an integral semi-positive definite
lattice, and let $R=\big\{x\in L:\bif(x,L)=0\big\}$ be its {\em
  radical}. The quotient $L/R$ is torsion-free (if $nx\in R$ for some
integral $n>0$, then $\bif(nx,L)=0$ implies $\bif(x,L)=0$, i.e.~$x\in
R$), and $\bif$ factors through a bilinear form~$\underline\bif$ on
$L/R$. The lattice $\lat /R:=(L/R,\underline \bif)$ is positive
definite (which follows from Silverster's theorem on using that
$\lat/R$ is by construction non-degenerate and semi-positive
definite).
\begin{Proposition}
  \label{prop:semi-definite-index}
  Let $\lat=(L,\bif)$ be an integral semi-positive definite lattice
  and let $R$ be its radical. For $k$ in $\frac12\Z$ and an integer
  $h$ denote by $\mathfrak{J}_{k,\lat}(\ve^h)$ the space of
  holomorphic functions on $\HP\times (\C\otimes L)$
  satisfying~\eqref{it:G-invariance},~\eqref{it:L-invariance} of
  Definition~\ref{def:Jacobi-forms-of-lattice-index}. The canonical
  projection $\alpha:\lat\rightarrow \lat/R$ defines an isomorphism
  \begin{equation*}
    \alpha^*:\mathfrak{J}_{k,\lat/R}(\ve^h)
    \rightarrow
    \mathfrak{J}_{k,\lat}(\ve^h)
    ,\quad
    \alpha^*(\phi)(\tau,z)=\phi(\tau,\alpha(z))
    .
  \end{equation*}
\end{Proposition}
\begin{proof}
  The only non-obvious fact is that $\alpha^*$ is surjective. For this
  let~$\psi$ be an element of $\mathfrak{J}_{k,\lat}(\ve^h)$. We have
  to show that $\psi(\tau,z)$ depends on~$z$ only modulo~$\C\otimes
  R$. Indeed, fix $\tau$ and $z$ and consider the function
  $f(w):=\psi(\tau,z+w)$ for $w$ in $\C\otimes R$ ($\subseteq
  \C\otimes L$). The transformation law~\eqref{it:L-invariance}
  implies than $f(w)=f(w+x\tau+y)$ for all $x$ and $y$ in $R$. In
  other words $f$ factors through a holomorphic function on the
  complex torus $\C\times R/\tau\otimes R+R$, and thus reduces to a
  constant.
\end{proof}

We conclude this section by explaining the relation between the
lattice index and matrix index Jacobi forms. If $F$ is a symmetric
positive definite $n\times n$ matrix with entries in $\frac12\Z$, we
set
\begin{equation*}
  J_{k,F}(\chi)
  =
  J_{k,\lat_F}(\chi)
  ,
\end{equation*}
where we use
\begin{equation*}
  \lat_F=\big(\Z^n,(x,y)\mapsto 2xFy^t\big)
  ,
\end{equation*}
and identify $\C^n$ with $\C\otimes_\Z\Z^n$ via the map
$\C\times\Z^n\mapsto \C^n$, $(z,x)\mapsto zx$.  The
spaces~$J_{k,F}(\chi)$ coincide with the spaces of Jacobi forms of
matrix index $F$ as they were, for half integral $F$\footnote{A
  symmetric matrix $F$ is called {\em half integral} if its entries
  are all in $\frac12\Z$ and its diagonal entries are integral, or
  equivalently, if $\lat_F$ is even.} considered in~\cite{Arakawa},
\cite{Boylan-Thesis}, \cite{Clery-Gritsenko},
and~\cite{Skoruppa-Schiermonnikoog}. It is clear that any integral
lattice $\lat$ is isometric to a lattice $\lat_F$ for a suitable $F$,
and then the spaces $J_{k,\lat}(\chi)$ and $J_{k,F}(\chi)$ are mapped
isomorphically to each other via the corresponding
map~\eqref{eq:alpha-star}. Thus the theories of Jacobi forms with {\em
  lattice index} as introduced here and the usual theory of Jacobi
forms with matrix-index are equivalent.

\section{Jacobi forms as theta functions}
\label{sec:theta-expansion}

In this section we show how to view Jacobi forms as vector-valued
modular forms. The advantage of this point of view is that we can
apply well-known results from the theory of elliptic modular forms to
rapidly deduce corresponding statements for Jacobi forms. This will
e.g.~be the strategy of the next section for obtaining dimension
formulas for the spaces $J_{k,\lat}(\ve^h)$.
 
As another application we describe a natural isomorphism between the
spaces $J_{k,\lat}(\chi)$ for singular weight ($k=\frac n2$) and
critical weight ($k=\frac{n+1}2$) and the spaces of $G$-invariant
vectors in $\Theta(\lat)\otimes\C(\chi)$, where $\Theta(\lat)$ denotes
a natural $G$-module associated to $\lat$. Here $n$ denotes the rank
of $\lat$.

For relating Jacobi forms to vector-valued modular forms we need some
preparations.  For $k$ in $\frac12\Z$, let $M_k\big(\Gamma(N)\big)$ be
the space of modular forms of weight $k$ on the principal congruence
subgroup $\Gamma(N)$ (the kernel of the map `reduction modulo $N$' on
$\SL$), where we assume that $N$ is divisible by $4$ if $k$ is not an
integer. We set
\begin{equation*}
  M_k=\bigcup_{N\geq 1}M_k\big(\Gamma(N)\big)
  .
\end{equation*}
The space $M_k$ becomes a $G$-module via the left action
\begin{equation*}
  (\alpha,h)\mapsto(h\mid_k\alpha^{-1})(\tau)=h(A\tau)\,w(\tau)^{-2k}
  \qquad (A=(\alpha,w))
  .
\end{equation*}
This statement is, for half integral $k$, not quite obvious. Indeed,
the transformation law for an $f$ in $M_k\big(\Gamma(N)\big)$ is
$f|_k(\gamma,w) = \kappa(\gamma)^{2k}\,f$ for all $\gamma$ in
$\Gamma(N)$, where $\kappa$ is the linear character of the inverse
image in $G$ (under the projection on the first component) of
$\Gamma_0(4)$ defined by
\begin{equation}
  \label{eq:Hecke-character}
  \kappa((A,w))=\theta(A\tau)/\theta(\tau)w(\tau)
  ,
\end{equation}
where $\theta=\sum_{r\in\Z}q^{r^2}$ (see~\cite[eq.~(1.10)]{Shimura} or
\cite[p.~287]{Hecke-1}).
One has to show that~$\kappa(\gamma)$ depends only on the
$\SL$-conjugacy class of~$\gamma$, which we leave to the reader.

Next, to an integral positive definite lattice $\lat=(L,\bif)$, we
associate a certain finite dimensional $G$-module $\Theta(\lat)$,
which we shall now describe.  Let $W(\lat)$ denote the space of maps
$\lambda:\sh L\rightarrow\C$ such that
$\lambda(r+x)=e(\bif(x))\,\lambda(r)$ for all $r$ in $\sh L$ and~$x$
in $L$. For $\lambda$ in $W(\lat)$ define a function on $\HP\times
(\C\otimes L)$ by
\begin{equation}
  \label{eq:basic-theta-functions}
  \vartheta_{\lat,\lambda}(\tau,z)
  =
  \sum_{r\in\sh L}
  \lambda(r)\,
  q^{\bif(r)}e\big(\bif(r,z)\big)
  ,
\end{equation}
and set
\begin{equation*}
  \Theta(\lat)
  :=
  \big\{
  \vartheta_{\lat,\lambda}:\lambda\in W(\lat)
  \big\}
  .
\end{equation*}
It is easy to verify that $\lambda\mapsto \vartheta_{\lat,\lambda}$
defines an isomorphism of $W(\lat)$ and~$\Theta(\lat)$. Note that
$W(\lat)$ has dimension $|\sh L/L|=\det(\lat)$.
\begin{Proposition}
  \label{prop:KloostermanX}
  The application $(\alpha,\phi)\mapsto \phi|\alpha^{-1}$ defines a
  $G$-module structure on $\Theta(\lat)$.  Here $\phi|\alpha$, for
  $\alpha=(A,w)$ in $G$, is defined like the left hand side
  of~\eqref{it:G-invariance} in
  Definition~\ref{def:Jacobi-forms-of-lattice-index} with $k$ replaced
  by $n/2$.
\end{Proposition}
The proposition is well-known for even
lattices~\cite{Kloosterman-I}\footnote{In fact,~\cite{Kloosterman-I}
  treats to a certain extent also the case of odd lattices; the
  translation of his results into our language seems to be somewhat
  tedious though.}. It is not hard to reduce the case of an odd
lattice to the case of even ones. We shall give below its simple proof
and the proof of the following supplement to the proposition.

\begin{Supplement}[to Proposition~\ref{prop:KloostermanX}]
  Let $\ell$ be the level of $\lat$. For every~$\gamma$ in $G$ with
  first component in $\Gamma(\ell)$ and every $\phi$ in
  $\Theta(\lat)$, one has $\phi|\gamma = \kappa_n(\gamma)\phi$. Here
  $n$ is the rank of $\lat$, and $\kappa_n(\gamma)=\kappa(\gamma)^n$
  if $n$ is odd (with the linear character $\kappa$
  from~\eqref{eq:Hecke-character}) and $\kappa_n(\gamma)=1$ if $n$ is
  even.
\end{Supplement}
We remark that the supplement makes indeed sense since the level of a
lattice of odd rank is divisible by~$4$ so that $\Gamma(\ell)\subseteq
\Gamma_0(4)$ if $n$ is odd.

It will be useful to note the two following formal properties of the
application $\lat\mapsto \Theta(\lat)$. If $\lat'$ is a sublattice of
$\lat$, then $\sh L$ is a union of $L'$-cosets in $\sh{L'}$, so that
every $\vartheta_{\lat,\lambda}$ is in $\Theta(\lat')$, which contains
hence $\Theta(\lat)$ as $G$-submodule. The second useful property is
that we can identify the $G$-modules
$\Theta(\lat)\otimes\Theta(\lat')$ and $\Theta(\lat\lsum\lat')$. The
identifying map is induced by the bilinear map $(\phi,\psi)\mapsto
\phi\otimes\psi$, where we use again
$\{\phi\otimes\psi\}(\tau,z_1\oplus z_2)=\phi(\tau,z_1)\psi(\tau,z_2)$
for $z_1\oplus z_2$ in $(\C\otimes L)\oplus(\C\otimes
L')=\C\otimes(L\oplus L')$.

Finally, for a linear character $\chi$ of $G$, we use $\C(\chi)$ for
the $G$-module with underlying vector space $\C$ and the $G$-action
$(\alpha,z)\mapsto \chi(\alpha) z$.  We can now formulate the main
result of this section.
\begin{Theorem}
  \label{thm:a}
  For any $k$ in $\frac12\Z$, any integral positive definite lattice
  $\lat$ of rank $n$ and any integer~$h$, there is a natural
  isomorphism
  \begin{equation*}
    J_{k,\lat}(\ve^h)
    \isom
    \big(M_{k-\frac n2}\otimes \Theta(\lat)\otimes\C(\ve^h)\big)^G
    .
  \end{equation*}
\end{Theorem}
\begin{Remark}
  As the proof will show we can replace the space $M_{k-n/2}$ on the
  right hand side of the isomorphism by
  $M_{k-n/2}\big(\Gamma(N)\big)$, where $N$ is any multiple of the
  level of~$\lat$ and $24/\gcd(h,24)$.
\end{Remark}
\begin{proof}[Proof of Proposition~\ref{prop:KloostermanX} and Supplement]
  The proposition is well-known for even
  lattices~\cite{Kloosterman-I}. Indeed, the group $G$ is generated by
  $t=(\mat 11{}1,1)$ and $s=(\mat {}{-1}1{},\sqrt \tau)$, and
  translating~\cite[eqs.~(1.9), (1.12)]{Kloosterman-I} into our
  terminology\footnote{If $\lat=(L,\bif)$ is an even lattice, $Q$ the
    Gram matrix with respect to a given $\Z$-basis of $L$, and $w(z)$
    the element of $\C\otimes L$ having $z$ as coordinates with
    respect to the given basis, then Kloosterman's function
    $\vartheta_{00}(z|\tau,a,1)$ (associated to the matrix $Q$ as
    in~\cite[eq.~(1.1)]{Kloosterman-I}) equals our
    $\vartheta_{\lat,\delta_a}(\tau, w(z))$, where $\delta_a(r) = 1$
    if $r\equiv w(a)/\det(Q)\bmod L$ and $\delta_a(r) = 0$
    otherwise. For this note that Kloosterman's special vectors are
    the coordinates of the elements of $\det(Q)\dual L$ with respect
    to the chosen basis.}
  \begin{equation*}
    \vartheta_{\lat,\lambda}|t^{-1}=\vartheta_{\lat,t.\lambda},
    \quad
    \vartheta_{\lat,\lambda}|s^{-1}=\vartheta_{\lat,s.\lambda}
    ,
  \end{equation*}
  where
  \begin{equation}
    \label{eq:s-t-action}
    \begin{split}
      \{t.\lambda\}(x)
      &=e(-\beta(x))\lambda(x),\\
      \{s.\lambda\}(x) &= \frac{e_8(n)}{\sqrt{\det(\lat)}} \sum_{y\in
        \sh L/L}e(\beta(x,y))\,\lambda(y) .
    \end{split}
  \end{equation}

  If $\lat$ is odd then $\Theta(\lat)$ is contained in
  $\Theta(\lat_1)$, where $\lat_1=(\ev L,\beta)$. It thus suffices to
  note that $t.\lambda$ and $s.\lambda$, for any $\lambda$ in $W(L)$
  is again contained in $W(L)$, where we view $W(\lat)$ as a subspace
  of $W(\lat_1)$ by continuing a $\lambda$ in $W(\lat)$ to a map on
  $\dual{\ev L}$ setting it equal to $0$ outside $\sh L$, and where
  $t.\lambda$ and $s.\lambda$ are given by the
  formulas~\eqref{eq:s-t-action}. Strictly speaking,
  $\{s.\lambda\}(x)$ is given, for $x$ in $\dual{\ev L}$, by the
  second formula of~\eqref{eq:s-t-action} with $\sh L/L$ replaced by
  $\sh L/\ev L$ and $\det (\lat)$ replaced by $\det(\lat_1) =
  4\det(\lat)$.  But then $\{s.\lambda\}(x)=0$ for $x$ outside $\sh
  L$. And for $x$ in $\sh L$, the terms $e(\bif(x,y))\,\lambda(y)$
  depend only on $x$ modulo~$L$, and thus the formula for
  $\{s.\lambda\}(x)$ can be written in the form~\eqref{eq:s-t-action}.

  Finally, for the supplement we remark that the statement is somehow
  well-known for even $\lat$ (the case of even $\lat$ with odd level
  is covered by \cite[Thm.~1]{Kloosterman-I}, the case of the
  nullwerte $\psi(\tau,0)$ of the functions in $\Theta(\lat)$ is
  treated in~\cite[Prop.~2.1]{Shimura}).  For an arbitrary even $\lat$
  we refer the reader to~\cite[Lemma~5.6]{Stroemberg}\footnote{For
    applying the reference one has to identify the $G$-module
    $\Theta(\lat)$ with the Weil representation of $G$ associated to
    the finite quadratic module $(\dual L/L, x+L\mapsto \bif(x)+\Z)$,
    which is easily done by comparing the action~\eqref{eq:s-t-action}
    of $s$ and $t$ to their action under the Weil representation.}. The
  case of an odd $\lat$ can be reduced to the case of an even lattice
  by using that $\Theta(\lat)$ is contained in $\Theta(\lat_1)$ and
  that the levels of $\lat$ and $\lat_1$ coincide.
\end{proof}

For the proof of Theorem~\ref{thm:a} we need the following
\begin{Lemma}
  \label{lem:epsilon-kappa}
  Let $h$ be an integer and $g$ be the greatest common divisor of $h$
  and~$24$. Then one has $\ve(\gamma)^h=\kappa_h(\gamma)$ for all
  $\gamma$ with first component in~$\Gamma(24/g)$.
\end{Lemma}
\begin{proof}
  We consider the function $f(\tau)=\eta(\tau)/\theta(\tau)$ with
  $\theta$ as in~\eqref{eq:Hecke-character}.  The function is a
  modular function (in fact, a modular unit) on some congruence
  subgroup of~$\SL$ (for a proof of this one can write
  $f(\tau)=i\eta(\tau)/q^{1/4}\vartheta(2\tau,\tau+1/2)$, or, using
  the Jacobi triple product~\eqref{eq:vartheta-product-expansion},
  $f=\left(\eta(\tau)/\eta(2\tau)\right) q^{1/12}\prod_{n\ge
    1}(1+q^{2n-1})^{-2}$
  and apply~\cite[Thm.~6.1]{Eholzer-Skoruppa}). Hence
  the stabilizer $\Gamma_h$ of $f^h$ is a congruence subgroup which
  contains the matrix $\mat 1{24/g}{}1$. According
  to~\cite[Thm.~2]{Wohlfahrt} it contains thus~$\Gamma(24/g)$. The
  lemma is now obvious.
\end{proof}

\begin{proof}[Proof of Theorem~\ref{thm:a}]
  The claimed isomorphism (from the right to the left) is given by the
  application
  \begin{equation}
    \label{eq:natural-iso}
    \sum_{\lambda \in B}
    h_\lambda \otimes \vartheta_{\lat,\lambda} \otimes c_\lambda
    \mapsto
    \sum_{\lambda \in B}
    h_\lambda\, \vartheta_{\lat,\lambda}\,c_\lambda
    ,
  \end{equation}
  where $B$ is a basis for $W(\lat)$. Note that the image of this map
  is indeed a Jacobi form. The verification of the
  axioms~\eqref{it:G-invariance}--\eqref{it:expansion} in the
  definition of Jacobi forms
  (Definition~\ref{def:Jacobi-forms-of-lattice-index}) is
  straightforward. For the transformation law~\eqref{it:L-invariance}
  with respect to $L\times L$ one verifies directly from the
  definition of the $\vartheta_{\lat,\lambda}$ that it holds for each
  of them.  This application is injective since, for any fixed $\tau$,
  the $\vartheta_{\lat,\lambda}(\tau,\cdot)$ as functions
  on~$\C\otimes L$ are linearly independent, when $\lambda$ runs
  through $B$ (so that the image of the given map is zero only if all
  $c_\lambda h_\lambda$ are identically zero).

  For showing that it is surjective we show that every $\phi\in
  J_{k,\lat}(\ve^h)$ can be expanded in terms of the basic theta
  functions $\vartheta_{\lat,\lambda}$. Indeed, for $x\in L$ and
  using the transformation law~\eqref{it:L-invariance} we have
  \begin{equation*}
    \phi(\tau,z+x\tau)\,e\big(\tau \bif(x) + \bif(x,z)\big)
    =
    e(\bif(x))\,\phi(\tau,z)
    .
  \end{equation*}
  Inserting the Fourier expansion of $\phi$ given
  in~\eqref{it:expansion} of
  Definition~\ref{def:Jacobi-forms-of-lattice-index}, this identity
  becomes
  \begin{multline*}
    \sum_{n\in \frac h{24} +\Z}
    \sum_{\begin{subarray}{c} r\in \sh{L}\\
        n \ge \bif(r)
      \end{subarray}} c(n,r) \, q^{n+\bif(r+x)-\bif(r)} \,
    e\big(\bif(r+x,z)\big)
    \\
    = \sum_{n\in \frac h{24} +\Z}
    \sum_{\begin{subarray}{c} r\in \sh{L}\\
        n \ge \bif(r)
      \end{subarray}} e(\bif(x))\,c(n,r) \, q^n \,
    e\big(\bif(r,z)\big) .
  \end{multline*}
  On comparing Fourier coefficients, we obtain
  \begin{equation}
    \label{eq:fourier-coeffs-comp}
    c(n,r)
    =
    e(\bif(x))\,
    c\big(n+\bif(r+x)-\bif(r),r+x\big)
    .
  \end{equation}
  If we set for any pair $D,r$ with $r\in\sh L$ and $D+\bif(r)\in\frac
  h{24}+\Z$,
  \begin{equation*}
    C(D,r)=c(D+\bif(r),r)
    ,\end{equation*}
  then the last identity can be restated in the form
  \begin{equation*}
    C(D,r)
    =
    e(\bif(x))\,C(D,r+x)
  \end{equation*}
  for all $x$ in $L$. In other words, for fixed $D$, the map $r\mapsto
  C(D,r)$ defines an element in $W(\lat)$, where we set $C(D,r)=0$ if
  $D+\bif(r)$ is not in $\frac h{24}+\Z$.  Hence we can write
  \begin{equation*}
    \phi(\tau,z)=\sum_{r\in \sh {L}}
    \sum_{\begin{subarray}{c} D\in \frac h{24}-\bif(r)+\Z\\
        D \ge 0
      \end{subarray}} C(D,r) \, q^{D+\bif(r)} \, e\big(\bif(r,z)\big)
    .
  \end{equation*}
  Let $B$ be a basis of $W(\lat)$, so that
  $C(D,\cdot)=\sum_{\lambda\in B}\gamma_\lambda(D)\lambda$ for
  suitable coefficients~$\gamma_\lambda(D)$. The last identity can
  then be written in the form
  \begin{equation*}
    \phi
    =
    \sum_{\lambda\in B}\Big(\sum_{\begin{subarray}{c} D\in \frac h{24}-\bif(r)+\Z\\
        D \ge 0
      \end{subarray}}\gamma_\lambda(D)q^D\Big)\vartheta_{\lat,\lambda}
    .
  \end{equation*}
  Denote the function defined by the inner sum by $h_\lambda(\tau)$.
  It remains to show that they are elements of $M_{k-n/2}$ (since
  then, from the injectivity of the application~\eqref{eq:natural-iso}
  and the invariance of $\phi$ under $G$, we deduce that $\sum_\lambda
  h_\lambda\otimes\vartheta_{\lat,\lambda}\otimes 1$ defines an
  element of the right hand side of the claimed isomorphism).

  From the supplement to Proposition~\ref{prop:KloostermanX} and the
  linear independence of the~$\vartheta_{\lat,\lambda}$ ($\lambda\in
  B$) as functions in the second argument for any fixed~$\tau$, we
  deduce
  $\ve^h(\gamma)\,h_\lambda=\kappa_n(\gamma)\,h_\lambda|_{k-n/2}\gamma$
  for all~$\gamma$ with first component in~$\Gamma(\ell)$, where
  $\ell$ is the level of~$\lat$. But by Lemma~\ref{lem:epsilon-kappa}
  $\ve(\gamma)^h=\kappa_h(\gamma)$ if the first component of $\gamma$
  is in $\Gamma\big(24/\gcd(h,24)\big)$. Moreover,
  $\kappa_h(\gamma)=\kappa_{2k}(\gamma)$ if $k\equiv h/2\bmod\Z$
  (since $\kappa$ takes values in $\{\pm
  1\}$~\cite[eq.~(1.10)]{Shimura}). It follows that $h_\lambda$ is in
  $M_{k-n/2}\big(\Gamma(N)\big)$ with $N$ as in the remark after the
  theorem (if $k\not\equiv h/2\bmod\Z$ then $\phi$ and all $h_\lambda$
  are zero).  This proves the theorem.
\end{proof}

We note some consequences of Theorem~\ref{thm:a}.  Following common
terminology we call two integral non-degenerate lattices $\lat_1$ and
$\lat_2$ stably equivalent (resp.~even stably equivalent) if there
exist integral unimodular lattices (resp.~even unimodular lattices)
$\lat [U]_1$ and $\lat [U]_2$ such that $\lat_1\lsum \lat [U]_1$ and
$\lat_2\lsum \lat [U]_2$ are isomorphic. It is
known~\cite[Cor.~1]{Wall-II} that $\lat_1$ and $\lat_2$ are stably
equivalent (respectively even stably equivalent) if the finite
bilinear modules $B_{\lat_1}$ and $B_{\lat_2}$ (resp.~the discriminant
modules $D_{\lat_1}$ and $D_{\lat_2}$) are isomorphic in the obvious
sense.  Here, for an integral lattice $\lat=(L,\bif)$, the {\em
  associated bilinear module $B_{\lat}$} is the finite abelian
group~$\dual L/L$ together with the induced bilinear map $\bif':\dual
L/L\times \dual L/L \rightarrow \Q/\Z$ (defined by
$\bif'(x+L,y+L)=\bif(x,y)+\Z$), i.e. the pair
\begin{equation*}
  B_{\lat}:=\big(\dual L/L,\bif')
  .
\end{equation*}
The {\em discriminant module $D_{\lat}$} of an even lattice $\lat$
consists of $\dual L/L$ together with the quadratic map $\bif'':\dual
L/L\rightarrow \Q/\Z$ induced by the quadratic form $x\mapsto \bif(x)$
defined by $\bif''(x+L)=\bif(x)+\Z$. Note that $\bif''$ is
well-defined only for even~$\lat$.

\begin{Theorem}
  \label{thm:stably-isomorphic-lattices}
  Let $\lat_1$ and $\lat_2$ be two stably equivalent integral positive
  definite lattices of rank $n_1$ and $n_2$, respectively. Then
  \begin{equation*}
    \Theta(\lat_1)\otimes \C(\ve^{3n_1})
    \isom
    \Theta(\lat_2)\otimes \C(\ve^{3n_2})
  \end{equation*}
  as $G$-modules.  If $\lat_1$ and $\lat_2$ are even stably equivalent
  then $\Theta(\lat_1)\isom \Theta(\lat_2)$.
\end{Theorem}
\begin{proof}
  As we shall show later, for any given integral non-degenerate
  lattice $\lat$, the formulas~\eqref{eq:s-t-action} define a
  $G$-module structure on $W(\lat)$. For a positive definite lattice
  this follows immediately from the fact that these formulas
  correspond (under the isomorphism $\lambda\mapsto
  \vartheta_{\lat,\lambda}$) to the action of $G$ on
  $\Theta(\lat)$. For a unimodular lattice $\lat [U ]=(U,\bif)$ of
  signature ~$s$ the set $\sh U/U$ consists of one element only, the
  space~$W(\lat [U ])$ is hence one-dimensional and it is immediate
  that the formulas~\eqref{eq:s-t-action} define a $G$-module
  structure on $W(\lat [U ])$ so that this space becomes as a
  $G$-module isomorphic to $\C(\ve^{-3s})$. For the latter statement
  we use that, for any $r$ in $\sh U$, we have $e(\bif(r))=e_8(s)$
  (see Lemma~\ref{lem:sigma-constant}).

  Suppose that there is an isomorphism $\alpha:\lat_1\lsum \lat
  [U]_1\rightarrow\lat_2\lsum \lat [U]_2$ for unimodular lattices
  $\lat [U]_1$ and $\lat [U]_2$ with, say, signatures $s_1$ and $s_2$.
  By what we saw we have $G$-module structures on $W(\lat_1)\otimes
  W(\lat [U]_1)$ and $W(\lat_2)\otimes W(\lat [U]_2)$ so that these
  $G$-modules are isomorphic to $W(\lat_1)\otimes \C(\ve^{-3s_1})$ and
  $W(\lat_2)\otimes \C(\ve^{-3s_2})$, respectively. But the map
  $\lambda_1\perp \lambda_2 \mapsto \lambda_1\lambda_2$ defines an
  isomorphism $W(\lat_1)\otimes W(\lat [U]_1)\rightarrow W(\lat_1\perp
  \lat [U]_1)$. If we transport the $G$-module structure via this
  isomorphism then the action of $G$ on $W(\lat_1\perp \lat [U]_1)$
  satisfies again the formulas~\eqref{eq:s-t-action} (with $\lat$
  replaced by $\lat_1\perp \lat [U]_1$). A similar reasoning defines a
  $G$-module structure on $W(\lat_2\perp \lat [U]_2)$. It is then
  easily checked from the formulas~\eqref{eq:s-t-action} that
  $\lambda\mapsto \lambda\circ \alpha$ defines an isomorphism of
  $G$-modules from $W(\lat_2\perp \lat [U]_2)$ onto $W(\lat_1\perp
  \lat [U]_1)$. This proves the first claimed isomorphism of the
  theorem with, however, $\ve^{3n_1}$ and $\ve^{3n_2}$ replaced by
  $\ve^{-3s_1}$ and $\ve^{-3s_2}$.  But $n_1+s_1=n_2+s_2$, which then
  implies the first part of the theorem.  For the second part we use
  that $n_1\equiv n_2\bmod 8$ if $\lat [U_1]$ and $\lat [U_2]$ are
  even unimodular, since then $s_1$ and $s_2$ are divisible by~$8$ as
  is well-known and can also be deduced from
  Lemma~\ref{lem:sigma-constant}.
\end{proof}
As an immediate consequence of this and Theorem~\ref{thm:a} we obtain

\begin{Theorem}
  \label{thm:stably-isomorphic-indices}
  Let $\lat_1$ and $\lat_2$ be two integral positive definite lattices
  of rank $n_1$ and $n_2$, respectively, whose associated bilinear
  modules are isomorphic. Then for every $k$ in $\frac12\Z$ and
  integer $h$, we have a natural isomorphism
  \begin{equation*}
    J_{k+\frac {n_1}2,\lat_1}\big(\ve^{h+3n_1}\big)
    \isom
    J_{k+\frac {n_2}2,\lat_2}\big(\ve^{h+3n_2}\big)
    .
  \end{equation*}
\end{Theorem}
\begin{Remark}
  If $\lat_1$ and $\lat_2$ are not only stably equivalent but even
  more even stably equivalent, we have $n_1\equiv n_2\bmod 8$ as we
  saw at the end of the proof of
  Theorem~\ref{thm:stably-isomorphic-lattices}.
\end{Remark}

There are two other interesting consequences of Theorem~\ref{thm:a}.
First of all, we note that $J_{k,\lat}(\ve^h)=0$ if $k<\frac n2$,
since there are no modular forms different than zero for negative
weight. For singular weight, i.e.~for~$k=\frac n2$ we find
\begin{Theorem}
  \label{thm:singular-iso}
  There is a natural isomorphism
  \begin{equation*}
    J_{\frac n2,\lat}(\ve^h)\isom\big(\Theta(\lat)\otimes \C(\ve^h)\big)^G
    ,\end{equation*}
  where $n=rank(\lat)$.
\end{Theorem}

For critical weight, i.e.~for $k=\frac{n+1}{2}$ we can still obtain an
explicit description of the corresponding spaces of Jacobi
forms. Namely, ~\cite[Thm.~A]{Serre-Stark} implies that the space
$M_{1/2}$ is generated by unary theta series. Using this one can prove
\begin{Theorem}
  \label{thm:d}
  For any positive integer $N$, we have
  \begin{equation*}
    M_{1/2}\big(\Gamma(4N)\big)
    =
    \bigoplus_{\begin{subarray}c
        N'|N\\
        N/N'\,\text{square-free}\end{subarray}}
    \nu^*\Theta\big(\lat [\Z](2N')\big)
    ,
  \end{equation*}
  where $\nu^*$ denotes the map $\phi(\tau,z)\mapsto \phi(\tau,0)$.
\end{Theorem}
\begin{proof}
  This theorem was proved with a slightly different formulation
  in~\cite[Thm.~5.2]{Skoruppa:Dissertation}.  More precisely, it was
  shown that, for any positive integer $N$, the $G$-module
  $M_{1/2}\big(\Gamma(4N)\big)$ is the direct sum of irreducible
  $G$-submodules each of which occurs in one of the spaces
  $\nu^*\Theta\big(\lat [\Z](2m)\big)$, where $m$ is a divisor
  of~$N$. From the supplement to Proposition~\ref{prop:KloostermanX}
  we know that each of the latter spaces is contained in
  $M_{1/2}\big(\Gamma(4N)\big)$. It follows that
  $M_{1/2}\big(\Gamma(4N)\big)$ is in fact equals to the sum of the
  spaces $\nu^*\Theta\big(\lat [\Z](2m)\big)$, where $m$ runs through
  the divisors of $N$.

  But $\nu^*\Theta\big(\lat [\Z](2m)\big)$ is a subspace of
  $\nu^*\Theta\big(\lat [\Z](2md^2)\big)$ for any positive integer~$d$
  (since $\Theta\big(\lat [\Z](2md^2)\big)$ contains
  $u_d^*\Theta\big(\lat [\Z](2m)\big)$, where $u_d$ is the isometric
  map $\lat [\Z](2m)$ into $\lat [\Z](2md^2)$ such that
  $u_d(x)=dx$). Accordingly, if $m|N$, and $d^2$ is the maximal square
  dividing $N/m$, then $\nu^*\Theta\big(\lat [\Z](2m)\big)$ is
  contained in $\nu^*\Theta\big(\lat [\Z](2N')\big)$, where $N':=md^2$
  divides $N$ such that $N/N'$ is square-free. Thus
  $M_{1/2}\big(\Gamma(4N)\big)$ equals indeed the sum of the spaces on
  the right hand side of the claimed decomposition.

  For proving that this sum is direct, we decompose each of the spaces
  $\Theta\big(\lat [\Z](2N')\big)$ into irreducible $G$-modules, and
  observe that two modules $\Theta\big(\lat [\Z](2N_j)\big)$ ($j=1,2$)
  have a $G$-irreducible component in common only if there are square
  divisors $t_j^2|N_j$ such that $N_1/t_1^2=N_2/t_2^2$, i.e.~only if
  $N_1/N_2$ is a rational square. For the latter and the decomposition
  into irreducible $G$-modules we refer the reader
  to~\cite[\S6]{Skoruppa-Schiermonnikoog}
  (or~\cite[Satz~1.8]{Skoruppa:Dissertation}). Finally, for two
  divisors $N_j$ of $N$ such that $N/N_j$ is square-free, the
  valuation of the quotient $N_1/N_2$ at any prime $p$ is $0$ or $1$,
  hence never a rational square.
\end{proof}

From Theorem~\ref{thm:a} (and the succeeding remark) and
Theorem~\ref{thm:d} we finally obtain as an immediate consequence the
description of the spaces of Jacobi forms of critical weight which we
announced at the beginning of this section.
\begin{Theorem}
  \label{thm:critical-iso}
  One has the decomposition
  \begin{equation*}
    J_{\frac{n+1}{2},\lat}(\ve^h)
    =
    \bigoplus_{\begin{subarray}c
        N'|N\\
        N/N'\, \text{square-free}\end{subarray}}
    \sigma^* J_{\frac{n+1}{2},\,\lat [\Z](2N')\lsum\lat}(\ve^h)
    ,
  \end{equation*}
  where $n$ is the rank of $\lat$, where $N$ is any multiple of the
  level of $\lat$ and $24/\gcd(24,h)$, and where $\sigma$, for any
  $N'$, is the isometric map $\sigma:\lat \rightarrow \lat
  [\Z](2N')\lsum\lat$, $\sigma(x)=0\oplus x$.
\end{Theorem}

\section{Dimension formulas}
\label{sec:dimension-formulas}

We may view an element of the right hand side of the isomorphism in
Theorem~\ref{thm:a} as a function $f$ on $\HP$ taking values in
$\Theta(\lat)\otimes\C(\ve^h)$. That~$f$ is invariant under the action
of $G$ can then be restated by saying that $f$ is a vector-valued
modular form of weight $k-n/2$ on~$G$. In this way the isomorphism of
Theorem~\ref{thm:a} identifies a space of Jacobi forms with a space of
vector-valued modular forms. Dimension formulas for spaces of
vector-valued modular forms have been deduced
in~\cite[Satz~5.1]{Skoruppa:Dissertation} (see
also~\cite[Thm~6]{Skoruppa-Schiermonnikoog}). Applying this general
formula to our situation we shall derive in this section an explicit
formula for the dimensions of the spaces $J_{k,\lat}(\ve^h)$.

For stating the final formula we need some preparations.  For an
integral non-degenerate lattice $\lat$ and an integer $t$, we set
\begin{equation*}
  \chi_{\lat}(t)={\sqrt{|\sh L/L|}^{-1}}\sum_{x\in\sh
    L/L}e\big(t\bif(x)\big)
  .
\end{equation*}
Note that $\bif(x)$, for $x$ in $\sh L$, depends only on the
coset~$x+L$.

Though we shall not pursue skew-holomorphic Jacobi forms in this
article it makes sense to introduce them at this point since they will
show up in the dimension formulas. For this let $\overline M_k$ denote
the space of anti-holomorphic modular forms. More precisely,
$\overline M_k$ denotes the space of all functions $f$ such that the
complex conjugate $\overline f$ lies in $M_k$ (which we introduced at
the beginning of Section~\ref{sec:theta-expansion}). The group $G=\Mp$
acts on $\overline M_k$ via $(\alpha,f)\mapsto \overline{\overline
  f|_k\alpha^{-1}}$. For an integral positive definite lattice
$\lat=(L,\bif)$, a $k$ in $\frac12\Z$ and an integer $h$, we define
the space $J_{k,\lat}^{\mathrm{skew}}(\ve^h)$ of {\it skew-holomorphic}
Jacobi forms as the space of functions $\phi$ on $\HP\times (\C\otimes
L)$ such that the application~\eqref{eq:natural-iso} defines an
isomorphism (from right to left)
\begin{equation}
  \label{eq:definition-of-skew-holomorphic-forms}
  J_{k,\lat}^{\mathrm{skew}}(\ve^h)
  \isom
  \big(\overline M_{k-\frac n2} \otimes \Theta(\lat) \otimes \C(\chi)\big)^G
  .
\end{equation}
We define the subspaces $J_{k,\lat}^{\text{skew, cusp}}(\ve^h)$ and
$J_{k,\lat}^{\text{skew, Eis}}(\ve^h)$ of cusp forms and Eisenstein
series so that they correspond under the given isomorphism to the
subspaces which one obtains by replacing $\overline M_{k-\frac n2}$ by
the subspaces of (anti-holomorphic) cusp forms and Eisenstein series,
and we define $J_{k,\lat}^{\mathrm{cusp}}(\ve^h)$ and
$J_{k,\lat}^{\mathrm{Eis}}(\ve^h)$ similarly.

\begin{Theorem}
  \label{thm:dimension-formula}
  For every $k$ in $\frac12\Z$, every integral positive definite
  lattice $\lat=(L,\bif)$ of rank $n$, and every integer $h$, one has
  $J_{k,\lat}(\ve^h)=0$ if $\pry:=k-h/2$ is not an integer. Otherwise
  one has
  \begin{equation*}
    \begin{split}
      \dim &J_{k,\lat}(\ve^h) - \dim J_{n+2-k,\lat}^{\text{skew,
          cusp}}(\ve^h)\
      \\
      &= \tfrac 1{24}\,(k - \tfrac n2 - 1)\, \left( \det(\lat) +
        (-1)^{\pry+n_2}\,2^{n-n_2} \right)
      \\
      &+ \tfrac 14 \sym{Re} \left( e_4(\pry)\,\chi_{\lat}(2) \right) +
      \tfrac 16 \leg{12}{2\pry+2n+1}
      \\
      &+ \tfrac {(-1)^\pry}{3\sqrt 3} \sym{Re} \left( e_6(\pry)
        e_{24}(n+2) \chi_{\lat}(-3) \right)
      \\
      &-\tfrac12 \sum_{x\in \sh L/L}\left\langle \tfrac h{24}-\bif(x)
      \right\rangle -\tfrac {(-1)^{\pry+n_2}}2
      \sum_{\begin{subarray}{c}
          x\in \sh L/L\\
          2x\in L
        \end{subarray}
      }\left\langle \tfrac h{24}-\bif(x) \right\rangle .
    \end{split}
  \end{equation*}
  Here $n_2$ is the rank of the unimodular constituent of the Jordan
  decomposition of~$\lat$ over $\Z_2$ (see the remark after the
  subsequent lemma), and $\langle x\rangle = x-\lfloor x \rfloor -
  1/2$.
\end{Theorem}
\begin{Remark}
  1.~The correction term, i.e.~the dimension of the space of
  skew-holo\-morph\-ic Jacobi forms, vanishes for $k\ge 2+\frac n2$, and
  $J_{k,\lat}(\ve^h)$ vanishes for $k<n/2$ (as follows
  from~\eqref{eq:definition-of-skew-holomorphic-forms} and
  Theorem~\ref{thm:a}). Therefore the theorem gives us an explicit
  formula for the dimension of $J_{k,\lat}(\ve^h)$ all weights $k$
  except for the two weights among the numbers $\frac n2$, $\frac
  n2+\frac12$, $\frac n2+1$ and $\frac n2+\frac 32$ which are
  congruent to $\frac h2$ modulo~$\Z$.

  2.~For $k=\frac n2$ and $k=\frac n2+\frac12$ we can use
  Theorems~\ref{thm:singular-iso} and~\ref{thm:critical-iso} to
  determine (the dimension of) $J_{k,\lat}(\ve^h)$, which amounts to
  determine the one-dimensional $\Mp$-submodules of 
  $\Theta(\lat')$ for certain lattices $\lat'$. This is a purely
  algebraic question which we shall eventually study in a subsequent
  publication.

  3.~For $k=\frac n2 + 1$, the correction term of the dimension
  formula refers to modular forms of weight $1$. For the dimension of
  these spaces there is in general no closed formula, and,
  accordingly, this term is in general unknown.

  4.~If $k=\frac n2 +\frac32$, the correction term refers to
  (anti-holomorphic) modular forms of weight $1/2$, and via
  Theorem~\ref{thm:d}
  and~\eqref{eq:definition-of-skew-holomorphic-forms} it leads again
  to the question of determining the dimension of the spaces of the
  $G$-invariants in certain $\Theta(\lat')$.

  5.~Note that the lattice $\lat$ enters the right hand side of the
  dimension formula only via its rank, its determinant, the quantity
  $n_2$ and the induced function $\bif:\sh L/L\rightarrow
  \Q/\Z$. Apart from the rank these data depend only on the
  discriminant module of $\ev{(\lat)}$, i.e.~the even-stable
  equivalence class of $\lat$.  The reason for this is the isomorphism
  of Theorem~\ref{thm:stably-isomorphic-indices}.
\end{Remark}

\begin{Lemma}
  \label{lem:weak-Jordan-decomposition}
  Let $\lat=(L,\bif)$ be a non-degenerate integral lattice.

  (i) Then there exists a sublattice $\lat'$ of odd index in $\lat$ of
  rank~$n$  such that
  \begin{equation*}
    \lat'\isom \lat_1\lsum \lat_2(2)
  \end{equation*}
  for suitable integral lattices $\lat_1$ and $\lat_2$ with $\lat_1$
  having odd determinant. If $\lat$ is odd, then $\lat_1$ can be
  chosen to be equal to $\lat [\Z](a_1)\lsum \cdots\lsum\lat
  [\Z](a_{n_2})$ for suitable odd integers~$a_j$.

  (ii) Let $n_2$ denote the rank of a lattice $\lat_1$ as in~(i). Then
  the number of elements of order~$2$ in $\sh L/L$ equals
  $2^{n-n_2}$. If $x+L$ is such an element, then
  $e\big(\bif(2x)\big)=(-1)^{n_2}$.
\end{Lemma}
\begin{Remark}
  Over $\Z_2$ every lattice possesses a Jordan decomposition, i.e.~it
  can be decomposed in the form $\lat [U]_0\lsum\lat [U]_1(2)\lsum\lat
  [U]_2(4)\lsum\cdots$, where the $\lat [U]_j$ are all unimodular. The
  ranks of the $\lat [U]_j$ do not depend on the specific
  decomposition. The lattice\footnote{For a lattice $\lat=(L,\bif)$ we
    use $\Z_2\otimes \lat$ for $(\Z_2\otimes L,\bif)$.}
  $\Z_2\otimes\lat_1$ can be taken as the constituent $\lat [U]_0$ of
  such a decomposition for $\Z_2\otimes \lat$. In particular, the rank
  $n_2$ of $\lat_1$ equals the rank of the unimodular part $\lat
  [U]_0$ in any Jordan decomposition of $\lat$ over $\Z_2$.
\end{Remark}
\begin{proof}[Proof of Lemma~\ref{lem:weak-Jordan-decomposition}]
  Let $\Z_{(2)}$ be the localization of $\Z$ at the prime ideal~$2\Z$,
  i.e.~the ring of rational numbers with odd denominator. There
  exists, with respect to~$\bif$, an orthogonal decomposition
  $\Z_{(2)}\otimes L=\bigoplus_j U_j$ into $\Z_{(2)}$-submodules $U_j$
  such that each $U_j$ is of rank one or two, and, moreover, such that
  $\bif(U_j,U_j)\subseteq 2\Z_{(2)}$ for all $U_j$ of rank two if
  there exists at least one $U_k$ of the form $\Z_{(2)}a$ with odd
  $\bif(a,a)$.

  (This holds true for any $R$-lattice $\lat [M] = (M,\bif)$ with
  $\bif(M,M)\subseteq R$ over any local dyadic ring $R$ with prime
  element $\pi$ as follows by a standard argument via induction over
  the rank of $M$. Indeed, let $I=R \pi^e$ be the $R$-ideal generated
  by all $\bif(x,y)$ with $x,y$ in $M$. If $M$ contains an element $a$
  with $\bif(a,a)=u\pi^e$ for some unit $u$ in $R$, then set
  $U=Ra$. Otherwise choose $a,b$ in $M$ such that $\bif(a,b)=\pi^eu$
  with a unit $u$ and set $U=Ra+Rb$. It follows $M=U \oplus U^\perp$,
  where $U^\perp$ means the orthogonal complement of $U$. If, after
  successively decomposing in this way, we have pieces, say, $U_1=R a$
  with $A:=\bif(a,a)$ being a unit, and $U_2=Rb+Rc$ of rank two with
  $(U_2,\bif)$ having odd determinant (and hence with $B:=\bif(b,b)$,
  $C:=\bif(c,c)$ in $R \pi$ and say $\bif(b,c)=1$), then we can
  replace such a contribution $U_1\oplus U_2$ by $R(a+b)\oplus R(a-A
  c) \oplus R\big((1-BC)a+(1+CA)b-(A+B)c\big)$, which is an orthogonal
  sum.)

  Let $f_j$ be the $\Z_{(2)}$-basis of $\Z_{(2)}\otimes L$ obtained by
  concatenating basis elements of all the $U_j$. If $e_j$ is a
  $\Z$-basis of $L$, then $(f_1,\ldots, f_n)=(e_1,\ldots,e_n)M$ for a
  suitable $M$ in $\sym{GL}(n,\Z_{(2)})$. Multiplying $M$ by an odd
  integer, if necessary, we can assume that $M$ has integral
  entries. Let $L'$ be the $\Z$-span of the~$f_j$. Then $L'$ is a
  subgroup of $L$ of index $\det(M)$, which is an odd
  integer. Renumbering the $f_j$ if necessary, we can assume that
  the~$f_j$ with $1\le j\le n_2$ are exactly those basis elements
  which generate the one-dimensional $U_j$ with odd
  $\bif(f_j,f_j)$. Let $L_1$ be the span of $f_1,\ldots, f_{n_2}$ and
  $L_2$ the span of $f_{n_2+1},\ldots f_n$, and set
  $\lat_1=(L_1,\beta|_{L_1})$ and $\lat_2=(L_2,\frac12\beta|_{L_2})$.
  Then $\lat'$, $\lat_1$ and $\lat_2$ are lattices as postulated.

  For proving~(ii) note that the number of elements of order~$2$ in
  $\dual L/L$ equals $2^l$, where $l$ is the number of even elementary
  divisors of any Gram matrix $G$ of $\lat$ (since $\dual L/L\isom
  \Z^n/G\Z^n$). But then, if $\lat'$, $\lat_1$ and $\lat_2$ are
  lattices as in (i), the number $l$ equals the rank of $\lat_2$,
  i.e.~$l=n-n_2$, where $n_2$ is the rank of $\lat_1$. If $\lat$ is
  even, then $\lat_1$ has even rank (since there exists a Gram matrix
  of $\lat_1$ which is a sum of $2\times2$ blocks as we saw in the
  proof of~(i)) and $\beta(2x)$ is integral if $2x$ is in $L$. This
  proves~(ii) for even $\lat$.

  If $\lat$ is odd, let $f_j$ be an orthogonal $\Z$-basis of
  $\lat_1$ (whose existence was shown in~(i)). Then, identifying
  $\lat_1$ with a sublattice of $\lat$, the element $r:=\frac
  12(f_1+\cdots+f_{n_2})$ is in $\sh L$ and its residue class
  modulo~$L$ is of order~$2$. The map $x\mapsto x+r$ induces a
  bijection between the set of elements of order $2$ in $\dual L/L$
  and $\sh L/L$, which proves the first statement of~(ii). For the
  second one note that $e\big(\beta(r)\big) = e\big(\sum_j \beta(f_j))
  = (-1)^{n_2}$, and that $e\big(\beta(r+x)\big)=e\big(\beta(r)\big)$
  for all elements $x$ of $\dual L$ such that $x+L$ has order~$2$.
\end{proof}

\begin{Lemma}
  \label{lem:sigma-constant}
  For an integral non-degenerate lattice $\lat$ of signature $s$, one
  has
  \begin{equation*}
    \chi_{\lat}(1)=\chi_{\ev{\lat}}(1)=e_8(s)
    .
  \end{equation*}
\end{Lemma}
\begin{proof}
  For even $\lat$, the claimed identity is known as Milgram's
  formula~\cite[App.~4, Thm.]{Milnor-Husemoller}. If $\lat$ is odd we
  write
  \begin{equation*}
    \sum_{x\in\dual{(\ev L)}/\ev L}e\big(\bif(x)\big)
    =
    \sum_{x\in\dual L/\ev L}e\big(\bif(x)\big)
    +
    2\sum_{x\in\sh L/L}e\big(\bif(x)\big)
    .
  \end{equation*}
  The first sum on the right is zero, as one sees by setting $x=y+u$,
  where~$y$ runs through representatives for $\dual L/L$ and $u$ for
  $L/\ev L$. Then $e\big(\bif(x)\big)$ becomes
  $e\big(\bif(y)+\bif(u)\big)$, and the sums over $u$ vanish since
  $u\mapsto e\big(\bif(u)\big)$ defines a non-trivial character of
  $L/\ev L$. Since $|\dual{(\ev L)}/\ev L| = 4\cdot |\dual L/L|$ the
  first identity of the theorem becomes now obvious.
\end{proof}

\begin{proof}[Proof of Theorem~\ref{thm:dimension-formula}]
  The fact that there are no non-zero Jacobi forms if $\pry$ is not
  integral, was already explained in the remarks after the definition
  of the space $J_{k,\lat}(\ve^h)$. We therefore assume in the rest of
  the proof that $\pry$ is integral.

  We deduce the formula of the theorem from the general dimension
  formula for vector-valued modular forms on $\SL$ as given in each
  of~\cite[p.~129]{Eholzer-Skoruppa},~\cite[Thm.~6]{Skoruppa-Schiermonnikoog}
  and which was proved in~\cite[Satz~5.1]{Skoruppa:Dissertation}. We
  repeat this formula here for the convenience of the reader in the
  form as given in~\cite{Eholzer-Skoruppa}:
  \begin{multline}
    \label{eq:general-dimension-formula}
    \dim M_k(V) - \dim M^{\mathrm{cusp}}_{2-k}(\Gdual V) =
    \tfrac 1{12}(k-1)\,\tr(1,V)\\
    +\tfrac14\sym{Re} \big(e_{4}(k)\tr(S^*,V)\big) +\tfrac 2{3\sqrt 3}
    \sym{Re}\big(e_{12}(2k+1)\tr(R^*,V)\big) -\sum_{j=1}^{\dim V}
    \langle \lambda_j\rangle .
  \end{multline}
  Here $S^*=(\mat {}{-1}1{},\sqrt\tau)$,
  $R^*=S^*T^*=(\mat{0}{-1}11,\sqrt{\tau+1})$, where $T^*=(\mat
  11{}1,1)$, and the $\lambda_j$ are rational numbers such that the
  characteristic polynomial of the operator in $\sym{GL}(V)$
  corresponding to $T^*$ equals
  $\prod_j\big(t-e(\lambda_j)\big)$. Moreover, $V$ is an arbitrary
  finite dimensional $G$-module such that a subgroup of finite index
  of~$G$ acts trivially on~$V$, and $\Gdual V$ is the $G$-module whose
  underlying space is the dual of~$V$ equipped with the $G$-action
  $(\alpha,\lambda)\mapsto \lambda\big(\alpha^{-1}\cdot\big)$.  The
  space $M_k(V)$ is the space of all holomorphic functions $f$ on
  $\HP$ taking values in $V$ such that
  (i)~$\{f|_h\alpha\}(\tau)=f(A\tau)v(\tau)^{-2k}=\alpha\{f(\tau)\}$
  for all $\alpha=(A,v)$ in $\Mp$ and all $\tau$ in $\HP$, and such
  that (ii)~$|f(\tau)|$ is bounded on all half planes (or any half
  plane) of the form $\sym{Im}(\tau)\ge v_0>0$. Here $|\cdot|$ denotes
  any norm on $V$ (the condition (ii) is independent of the particular
  choice of norm since all norms on $V$ are equivalent). We use
  $M^{\mathrm{cusp}}_k(V)$ for the subspace of those $f$ such that
  $f(\tau)$ tends to $0$ as $\sym{Im}(\tau)$ tends to infinity.  The
  dimension formula is valid for all weights $k$ in $\frac12\Z$ and
  all $V$ such that $(-1,i)$ acts as multiplication by $i^{-2k}$.

  The latter is, of course, no restriction. Indeed, if $V$ is an
  arbitrary finite dimensional $G$-module (factoring through a finite
  quotient of~$G$), then $V$ splits under the action of the center
  $Z(G)=\langle (-1,i)\rangle$ into parts $V=\oplus_\chi V(\chi)$,
  where $\chi$ runs through the four characters of $Z(G)$ and
  $V(\chi)$ denotes the subspace of $v$ in $V$ such that $\alpha v =
  \chi(\alpha)v$ for all $\alpha$ in $Z(G)$.  Accordingly, $M_k(V)$
  splits into a direct sum $\oplus_\chi M_k\big(V(\chi)\big)$. If $f$
  is an element of $M_k\big(V(\chi)\big)$, then $f|(-1,i) = i^{-2k}f$,
  by the very definition of the slash action, whereas
  $f|(-1,i)=(-1,i)f = \chi(-1,i)f$. It follows that
  $M_k\big(V(\chi)\big)=0$ unless $\chi$ equals $\chi_k:(-1,i)\mapsto
  i^{-2k}$, and then $M_k(V)=M_k\big(V(\chi_k)\big)$, and we can apply
  the formula~\eqref{eq:general-dimension-formula}. In other words,
  the formula~\eqref{eq:general-dimension-formula} is true for
  arbitrary $V$ (finite dimensional and with action of a finite
  quotient of $G$) if we replace on the right hand side
  $\tr(\alpha,V)$ by
  \begin{equation*}
    \tr_k(\alpha,V)
    :=
    \frac 14\sum_{t\bmod 4}i^{2kt}\,\tr\big(((-1,i)^t\alpha,V\big)
    ,
  \end{equation*}
  and where the $\lambda_j$ run through the eigenvalues of $T^*$
  acting on the subspace~$V(\chi_k)$.

  After these preparations we obtain a first dimension formula for
  $J_{k,\lat}(\ve^h)$ by
  specializing~\eqref{eq:general-dimension-formula} to
  $V=\Theta(\lat)\otimes \C(\chi)$ and applying
  Theorem~\ref{thm:a}.  For this note that we may
  identify $M_k(V)$ with the space of $G$-invariants in $M_k\otimes V$
  by identifying an element $\sum_j f_j\otimes v_j$ with the map
  $\tau\mapsto \sum_j f_j(\tau)v_j$.  Accordingly
  Theorem~\ref{thm:a} can be restated in the form
  \begin{equation*}
    J_{k,\lat}(\ve^h)\isom M_{k-n/2}\big(\Theta(\lat)\otimes \C(\ve^h)\big)
    .    
  \end{equation*}
  The correction term in ~\eqref{eq:general-dimension-formula} equals
  here the dimension of the subspace of $G$-invariant vectors in
  $M_{n/2+2-k}^{\mathrm{cusp}}\otimes\Gdual{\big(\Theta(\lat)\otimes
    \C(\ve^h)\big)}$, which equals the dimension of $G$-invariant
  elements in $\Gdual {(M_{n/2+2-k}^{\mathrm{cusp}})}\otimes
  \Theta(\lat)\otimes \C(\ve^h)$.  But $\Gdual
  {(M_{n/2+2-k}^{\mathrm{cusp}})}$ is isomorphic as $G$-module to
  $\overline M_{n/2+2-k}^{\mathrm{cusp}}$ (as can either be proved by
  decomposing spaces into finite dimensional pieces and comparing the
  traces on these pieces, or by mapping an element $f$ of $\overline
  M_{n/2+2-k}^{\mathrm{cusp}}$ to the functional $\langle\_,\overline
  f\rangle$ on $M_{n/2+2-k}^{\mathrm{cusp}}$, where
  $\langle\_,\_\rangle$ denote the suitably normalized Petersson
  scalar product on elliptic modular forms). The error term equals
  hence the dimension of $J_{n+2-k,\lat}^{\text{skew,cusp}}(\ve^h)$.

  We thus obtain a dimension formula as claimed, where, however, the
  right hand side equals the right hand side
  of~\eqref{eq:general-dimension-formula} with $k$ replaced by $k-n/2$
  and with $\tr(\alpha,V)$ replaced by
  \begin{equation*}
    \chi(\alpha)
    :=
    \frac 14
    \sum_{t\bmod 4}i^{(2k-n)t}\,
    \tr\big((-1,i)^t\alpha,\Theta(\lat)\otimes \C(\ve^h)\big)
    ,
  \end{equation*}
  and with the appropriate choice for the $\lambda_j$.  The element
  $(1,-1)$ acts as multiplication by $(-1)^n$ on $\Theta(\lat)$, and
  $\ve^h(-1,i)=i^{-h}$ (as follows directly from the definition of the
  action of $\Mp$ on $\Theta(\lat)$ as given in
  Proposition~\ref{prop:KloostermanX} and the
  definition~\eqref{eq:epsilon} of~$\ve$). Inserting this in the
  formula for $\chi(\alpha)$ we obtain
  \begin{equation}
    \label{eq:chi-general}
    \chi(\alpha)
    =
    \tfrac12 \ve(\alpha)^h
    \left(\tr\big(\alpha,\Theta(\lat)\big)
      +
      (-1)^\pry i^{-n}\,\tr\big((-1,i)\alpha,\Theta(\lat)\big)
    \right)
    ,
  \end{equation}
  where we also used that $2\pry=2k-h$ is even.

  For simplifying $\chi(1)$ we study the action of $(-1,i)$ on
  $\Theta(\lat)$. For this let $R$ be a set of representatives for
  $\sh L/L$. For each $r$ in $R$, let $\lambda_r$ be that element of
  $W(\lat)$ such that $\lambda_r(r)=1$ and $\lambda(s)=0$ if $s\not
  \equiv r \bmod L$. The $\lambda_r$ form a basis of $W(\lat)$, and
  accordingly the $\vartheta_{\lat,\lambda_r}$ (as defined
  in~\eqref{eq:basic-theta-functions}) form a basis of
  $\Theta(\lat)$. It is easily checked that
  \begin{equation*}
    \vartheta_{\lat,\lambda_r}|(-1,i)^{-1}
    =
    i^n\,e\big(\bif(r+r')\big)\,\vartheta_{\lat,\lambda_{r'}}
    ,    
  \end{equation*}
  where $r'$ denotes that element in $R$ such that $r'\equiv -r \bmod
  L$.

  From this we recognize that only those $r$ in $R$ contribute to the
  trace of $(1,i)$ on~$\Theta(\lat)$ which satisfy $r\equiv r' \bmod
  L$, or, equivalently, $2r\in L$. Using
  Lemma~\ref{lem:weak-Jordan-decomposition} we therefore obtain
  $\tr\big((-1,i),\Theta(\lat)\big) = i^n(-1)^{n_2}2^{n-n_2}$, and
  accordingly
  \begin{equation*}
    \chi(1)=\tfrac 12 \big(\det(\lat) + (-1)^{\pry+n_2}2^{n-n_2}\big)
    .
  \end{equation*}
  Inserting this for $\tr(1,V)$
  in~\eqref{eq:general-dimension-formula} yields the first term on the
  right of the claimed formula.

  For computing $\chi(S^*)$ we use that
  $(-1,i)S^*=(S^*)^3=(1,-1)(S^*)^{-1}$, and that~$(1,-1)$ acts on
  $\Theta(\lat)$ as multiplication by $(-1)^n$ so that
  \begin{equation*}
    \begin{split}
      \chi(S^*) &= \tfrac12 \ve(S^*)^h \left(
        \tr\big(S^*,\Theta(\lat)\big) + (-1)^\pry
        i^n\tr\big((S^*)^{-1},\Theta(\lat)\big) \right)
      \\
      &= \ve(S^*)^h
      i^{-\pry}e_8(n)\sym{Re}\Big(i^{\pry}e_8(-n)\,\tr\big(S^*,\Theta(\lat)\big)\Big).
    \end{split}
  \end{equation*}
  But $\tr\big(S^*,\Theta(\lat)\big)=e_8(n)\chi_{\lat}(2)$
  (cf.~\eqref{eq:s-t-action}), and $\ve(S^*) = e_8(-1)$ and hence
  \begin{equation*}
    \chi(S^*)
    =
    e_4(n/2-k)\sym{Re}\big(i^{\pry}\chi_{\lat}(2)\big)
    .
  \end{equation*}
  Inserting this for $\tr(S^*,V)$
  in~\eqref{eq:general-dimension-formula} (and replacing
  in~\eqref{eq:general-dimension-formula} $k$ by $k-n/2$) we obtain
  the second term on the right of the claimed formula.

  For simplifying the contribution $\chi(R^*)$ we use the identities
  \begin{equation*}
    \tr\big(R^*,\Theta(\lat)\big) = e_4(n),\quad
    \tr\big((-1,i)R^*,\Theta(\lat)\big)=e_8(3n)\chi_{\lat}(-3),
  \end{equation*}
  and $\ve(R^*)=e_{12}(-1)$. The first two can be deduced by writing
  $R^*=S^*T^*$, and applying the formulas~\eqref{eq:s-t-action} for
  the action of $S^*$ and $T^*$ and the formula
  $\theta_{\lat,\lambda}|(-1,i)^{-1}=\theta_{\lat,\lambda'}$, where
  $\lambda'(r)=i^n\lambda(r)$. We therefore obtain
  \begin{equation*}
    \begin{split}
      \chi(R^*) &= \tfrac12 e_{12}(-h) \big(
      e_8(n)\chi_{\lat}(1)e_4(n) + (-1)^\pry e_8(n)\chi_{\lat}(-3)
      \big) .
    \end{split}
  \end{equation*}
  Using $\chi_{\lat}(1)=e_8(n)$ (Lemma~\ref{lem:sigma-constant}), we
  obtain accordingly,
  \begin{multline*}
    \tfrac 2{3\sqrt 3} \sym{Re}\left(e_{12}(2k-n+1)\chi(R^*)\right)\\
    = \tfrac 1{3\sqrt 3} \sym{Re} \big((e_{12}(2\pry+2n + 1) \big) +
    \tfrac {(-1)^\pry}{3\sqrt 3} \sym{Re} \left( e_6(\pry) e_{24}(n+2)
      \chi_{\lat}(-3) \right) .
  \end{multline*}
  It is easily verified that the first term equals
  $\frac16\leg{12}{2\pry+2n+1}$. We thus recognize the third and the
  fourth term on the right of the claimed dimension formula.

  Finally, a basis for the subspace of all elements $v$ in
  $\Theta(\lat)\otimes\C(\ve^h)$ satisfying $(-1,i)v=i^{-(2k-n)}v$ is
  given by
  \begin{equation*}
    v_r
    :=
    \vartheta_{\lat,\lambda_r}\otimes 1
    +
    (-1)^\pry\,e\big(\bif(r+r')\big)\,\vartheta_{\lat,\lambda_{r'}}\otimes 1
    \qquad
    (r\in R')
    .
  \end{equation*}
  Here we use the $\lambda_r$ as introduced above, and $R'$ is a
  subset of $R$ representing all elements in $R$ modulo the involution
  $r\mapsto r'$, omitting or including the orbits of elements $r$ such
  that $2r\in L$ accordingly as $\pry+n_2$ is odd or even. For the
  later condition note that $v_r$ vanishes if $2r\in L$ and $\pry+n_2$
  is odd since, by Lemma~\ref{lem:weak-Jordan-decomposition}, indeed,
  $e\big(\bif(r+r')\big)=(-1)^{n_2}$.  That $(-1,i)v_r=i^{-(2k-n)}v_r$
  follows from $\ve^h(-1,i)=i^{-h}$ and the transformation formula for
  $\vartheta_{\lat,\lambda_r}$ under $(-1,i)^{-1}$ as given above. One
  has $T^*v_r=e\big(\frac h{24}-\beta(r)\big)v_r$.  The claimed
  dimension formula is now obvious.
\end{proof}

As a corollary to the proof of the dimension formula we obtain

\begin{Supplement}[to Theorem~\ref{thm:dimension-formula}]
  For all integers $k$, one has
  \begin{multline*}
    \dim J_{k,\lat}^{\mathrm{Eis}}(\ve^h) + \dim
    J_{n+2-k,\lat}^{skew,Eis}(\ve^h)
    =\\
    \tfrac12 \Big( \#\big\{x\in \sh L/L:\beta(x)\equiv h/24 \bmod
    \Z\big\} \\+ (-1)^{\pry+n_2} \#\big\{x\in \sh L/L:\beta(x)\equiv
    h/24 \bmod \Z,\ 2x\in L\big\} \Big) .
  \end{multline*}
\end{Supplement}
\begin{proof}
  Replacing in the dimension
  formula~\eqref{eq:general-dimension-formula} $k$ by $2-k$ and $V$
  by~$\Gdual V$, and adding the resulting identity to the original
  one, gives
  \begin{equation*}
    \dim M_k(V)-\dim M_k^{\mathrm{cusp}}(V)
    +
    \dim M_{2-k}(\Gdual V)-\dim M_{2-k}^{\mathrm{cusp}}(\Gdual V)
    =
    N(V)
    ,
  \end{equation*}
  where $N(V)$ denotes the number of $j$ such that $\lambda_j$ is an
  integer. If we let $V=\Theta(\lat)\otimes\C(\chi)$ the left hand
  side becomes the left hand side of the claimed identity as we saw in
  the proof of~\eqref{eq:general-dimension-formula}. For proving the
  claimed formula it remains to determine~$N(X)$ for the submodule $X$
  of those $v$ in $\Theta(\lat) \otimes \C(\ve^h)$ which satisfy
  $(-1,i)v=i^{-(2k-n)}v$. A basis of $X$ consisting of eigenvectors
  with respect to the action of $(\mat 11{}1,1)$ is given by the $v_r$
  ($r\in R')$ introduced at the end of the proof
  of~\eqref{eq:general-dimension-formula}. From $(\mat
  11{}1,1)v_r=e\big(\frac h{24}-\bif(r)\big) v_r$ we now recognize the
  claimed formula.
\end{proof}

We note that there are many cases in which we do not have any
Eisenstein series. The latter is true if the level of $\lat$ is
relatively prime to the denominator of $h/24$. It is also true if
$\lat$ is odd, $\ve^h$ is trivial, and $\bif(x)$ is not integral for
all shadow vectors~$x$, which is the case e.g.~for $\lat=\lat [\Z]^n$
with $n\not \equiv 0\bmod 8$. If $\lat$ is a maximal even lattice,
then there is exactly one Eisenstein series in $J_{k,\lat}(1)$ if $k$
is even and none if $k$ is odd.


  


\section{Structure theorems}
\label{sec:structure-theorems}

For calculating systematically explicit examples of Jacobi forms it is
useful to note that multiplication of a Jacobi form by a modular form
on $\SL$ yields a Jacobi form of the same index and character. More
precisely, for an integral positive definite lattice $\lat$ and an integer $h$,
set
\begin{equation*}
  \begin{split}
    J_{\mathrm{even},\lat}(\ve^h) &= \bigoplus_{\begin{subarray}{c}
        k\in\frac12\Z\\ k\equiv \frac h2\bmod 2\Z
      \end{subarray}
    } J_{k}(\ve^h),
    \\
    J_{\mathrm{odd},\lat}(\ve^h) &= \bigoplus_{\begin{subarray}{c}
        k\in\frac12\Z\\ k\equiv \frac h2+1\bmod 2\Z
      \end{subarray}
    } J_{k}(\ve^h)
  \end{split}
\end{equation*}
Then both spaces become graded vector spaces with respect to
multiplication by elements in the ring
\begin{equation*}
  M_*(1):=\bigoplus_{k\in 2\Z} M_k\big(\SL\big)=\C[E_4,E_6] ,
\end{equation*}
where $E_4$ and $E_6$ are the Eisenstein series on $\SL$ of weight $4$
and~$6$, respectively:
\begin{equation*}
  E_4=1+240\sum_{n\ge0}\sigma_3(n)\,q^n,
  \quad
  E_6=1-504\sum_{n\ge0}\sigma_5(n)\,q^n .
\end{equation*}
Recall that, for weights $k\not\equiv h/2 \bmod \Z$, the spaces
$J_{k,\lat}(\ve^h)$ vanish. Thus the $M_*(1)$-modules
$J_{\mathrm{even}}(\ve^h)$ and $J_{\mathrm{odd}}(\ve^h)$ comprise in fact
all Jacobi forms with index~$\lat$ and character $\ve^h$. The
transformation law for a Jacobi form of weight $k$ and
character~$\ve^h$ applied to the element $(-1,i)$ reads
$\phi(\tau,-z)i^{-2k}=(-i)^h\phi(\tau,z)$. From this we obtain
\begin{Proposition}
  The Jacobi forms in $J_{\mathrm{even},\lat}(\ve^h)$ and
  $J_{\mathrm{odd},\lat}(\ve^h)$ are exactly those Jacobi forms on
  $\lat$ with character $\ve^h$ which are even and odd in the lattice
  variable, respectively.
\end{Proposition}

\begin{Theorem}
  \label{thm:Jacobi-forms-as-module-over-CE4E6}
  Let $\lat$ be an integral positive definite lattice of rank $n$ and
  $h$ be an integer.  The $M_*(1)$-modules
  $J_{\mathrm{even},\lat}(\ve^h)$ and $J_{\mathrm{odd},\lat}(\ve^h)$ are
  free of ranks $\tfrac 12(\det(\lat) + (-1)^{n_2}\,2^{n-n_2})$ and
  $\tfrac 12(\det(\lat) - (-1)^{n_2}\,2^{n-n_2})$, respectively. (Here
  $n_2$ is the integer described in
  Lemma~\ref{lem:weak-Jordan-decomposition}).
\end{Theorem}
\begin{proof}
  For proving that $J_{\mathrm{odd|even},\lat}(\ve^h)$ is free over
  $M_*(1)$ one can mimic the proof of~\cite[Thm.~8.4]{Eichler-Zagier},
  where it is shown that $\bigoplus_{k,m\in\Z} J_{k,\Z(2m)}$ is free
  over $M_*(1)$. The statement concerning the rank follows
  from~\eqref{eq:dimensions-asymptotics} (which in turn is an
  immediate consequence of the dimension formula in
  Theorem~\ref{thm:dimension-formula}) and $\dim
  M_{2k}\big(\SL\big)=k/6+O(1)$.
\end{proof}

For doing explicit calculations it is usually useful to look at the
Poincar\'e-Hilbert series of $J_{\mathrm{odd|even},\lat}(\ve^h)$, which
provides further information about $J_{\mathrm{odd|even},\lat}(\ve^h)$
as module over $M_*(1)$. It is defined as the formal power series
\begin{equation*}
  P_{\mathrm{parity},\lat,h}(t)
  =
  \sum_{\begin{subarray}{c} k\in\frac12\Z\\ k\equiv \frac
      h2 + p\bmod 2\Z
    \end{subarray}
  } \dim J_{k}(\ve^h)\,t^k
  ,
\end{equation*}
where $\mathrm{parity}$ means $\mathrm{even}$ or $\mathrm{odd}$ and $p$ denotes $0$ in the first case and $1$ otherwise. The
Hilbert-Poincar\'e series of $M_*(1)$ is $1/(1-t^4)(1-t^6)$, from
which it follows that
\begin{equation*}
  P_{\mathrm{parity},\lat,h}(t)
  =
  \frac {Q_{\mathrm{parity},\lat,h}(t)}{(1-t^4)(1-t^6)}
\end{equation*}
for a polynomial $Q_{\mathrm{parity},\lat,h}(t)$ in $t^{1/2}$. Indeed,
if $\phi_j$ ($1\le j\le r$) is a basis for
$J_{\mathrm{parity},\lat}(\ve^h)$ as module over $M_*(1)$, and if
$\phi_j$ is a Jacobi form of weight $k_j$, then
\begin{equation*}
  Q_{\mathrm{parity},\lat,h}(t)=\sum_{j=1}^r t^{k_j}
  .
\end{equation*}
Thus the $t^k$-th coefficient of the polynomial
$Q_{\mathrm{parity},\lat,h}(t)$ is the number of basis elements of
weight $k$ of any given $M_*(1)$-basis of
$J_{\mathrm{parity},\lat}(\ve^h)$. The
polynomial~$P_{\mathrm{parity},\lat,h}(t)$ can be rapidly calculated
using the dimension formula of
Theorem~\ref{thm:dimension-formula}. Note that for a give parity and
$\lat$ and $h$ there is exactly one weight $s\equiv h/2\bmod 2\Z$ in
the range $n/2\le s < n/2+2$, where $n$ is the rank of $\lat$. The
weight $s$ is the smallest weight for which the corresponding
$s$-graded part in $J_{\mathrm{parity},\lat}(\ve^h)$ is not necessarily
zero, and it is also the only weight where the correction term in the
dimension formula for the $s$-graded part is not necessarily zero. In
particular, the polynomial~$P_{\mathrm{parity},\lat,h}(t)$ equals $t^s$
times a polynomial in $t^2$. For the calculation
of~$P_{\mathrm{parity},\lat,h}(t)$ it is useful to note that the degree
of $P_{\mathrm{parity},\lat,h}(t)$ is strictly less than $n/2+12$. In
other words, one has
\begin{Supplement}[to
  Theorem~\ref{thm:Jacobi-forms-as-module-over-CE4E6}]
  The weights $k$ of the elements of a (graded) basis of
  $J_{\mathrm{parity},\lat}(\ve^h)$ over $M_*(1)$ satisfy the inequality
  \begin{equation*}
    k < n/2 +12
    ,
  \end{equation*}
  where $n$ is the rank of $\lat$.
\end{Supplement}
\begin{proof}
  For $k\in h/2 + \Z$, denote the right hand side of the dimension
  formula in~Theorem~\ref{thm:dimension-formula} by $a(k)$, and set
  $b(k)=a(k)-a(k-4)-a(k-6)+a(k-10)$. Note that $b(k) = 0$. But for
  $k-10 \ge n/2 +2$ the numbers $a(k-j)$ ($j=0,4,6,10$) equal the
  dimension of $J_{k-j,\lat}(\ve^h)$, and hence $b(k)$ (for $k\equiv
  s\bmod 2\Z$) equals the $k$-th coefficient of
  $P_{\mathrm{parity},\lat,h}(t)$.
\end{proof}
Note that a similar reasoning gives also for other modules over
$M_*(1)$ a bound for the weight of the generators. For instance, one
sees similarly that the $M_*(1)$-module of (vector or scalar valued)
modular forms on a given subgroup of finite index in $\SL$ has always
generators whose weights are strictly less than~$12$.

\begin{table}[ht]
  \centering
  \caption{Hilbert-Poincar\'e polynomials for root lattices}
  \begin{tabular}[h]{@{}>{$}l<{$}>{$}l<{$}>{$}l<{$}@{}}
    \toprule
    \lat&Q_{\mathrm{even},\lat,0}&Q_{\mathrm{odd},\lat,0}\\
    \midrule
    \lat [A]_1& t^2 \left(t^{4} + t^{2}\right) & 0\\
    \lat [A]_2& t^2 \left(t^{4} + t^{2}\right) & t^{9}\\
    \lat [A]_3& t^2 \left(t^{6} + t^{4} + t^{2}\right) & t^{9}\\
    \lat [D]_4& t^2 \left(2t^{6} + t^{4} + t^{2}\right) & 0\\
    \lat [D]_5& t^4 \left(x^4 + x^2 + 1\right) & t^{7}\\
    \lat [E]_6& t^4 \left(t^{2} + 1\right) & t^{7}\\
    \lat [E]_7& t^4 \left(t^{2} + 1\right) & 0\\
    \bottomrule
  \end{tabular}
  \label{tab:HP-polys}
\end{table}

It is also useful to consider the spaces
\begin{equation*}
  J_{\mathrm{parity},\lat}(\ve^*)
  :=
  \bigoplus_{h\bmod 24}J_{\mathrm{parity},\lat}(\ve^h)
  .
\end{equation*}
Each of these two bigraded vector spaces is a bigraded module over the
bigraded ring
\begin{equation*}
  M_*(\ve^*)
  :=
  \bigoplus_{0\le h < 23}\eta^{h}M_*(1)
  .
\end{equation*}
The modules $J_{\mathrm{parity},\lat}(\ve^*)$ are in general no longer
free over $M_*(\ve^*)$ as we shall see in a moment in
Theorem~\ref{thm:modules-over-A2}, but their structure seems to be
still not to be too complicated. Namely, define
$J_{\mathrm{parity},\lat}^!(\ve^*)$ to be the sum of all spaces
$J_{\mathrm{parity},\lat}^!(\ve^h)$ with $h$ running through all
integers modulo~$24$, and where $J_{\mathrm{parity},\lat}^!(\ve^h)$ is
the sum of the spaces of weakly holomorphic Jacobi forms
$J_{k,\lat}^!(\ve^h)$, with $k$ running through all $k$ in $\frac12\Z$
such that $k-h/2$ is a `parity' integer. Usual multiplication of
functions defines on $J_{\mathrm{parity},\lat}^!(\ve^*)$ the structure
of a bigraded module over the graded ring $M_*^!(\ve^*)$ generated by
$E_4$, $E_6$, $\eta$ and $\eta^{-1}$ (the grading given by weight and
character). Multiplication by a power $\eta^l$ defines isomorphisms
$J_{\mathrm{parity},\lat}^!(\ve^h)\isom
J_{\mathrm{parity},\lat}^!(\ve^{h+l})$. In particular, if $\{\phi_j\}$
is a basis of $J_{\mathrm{parity},\lat}(\ve^h)$ over $M_*(1)$, then
$\{\phi_j\}$ generates $J_{\mathrm{parity},\lat}^!(\ve^*)$ over
$M_*^!(\ve^*)$, and using the bigrading it is easy to see that the
$\{\phi_j\}$ are linearly independent over~$M_*^!(\ve^*)$. We thus
obtain
\begin{Theorem}
  The $M_*^!(\ve^*)$-modules $J_{\mathrm{even},\lat}^!(\ve^*)$ and
  $J_{\mathrm{odd},\lat}^!(\ve^*)$ are free of rank
  \begin{equation*}
    \tfrac
    12(\det(\lat) + (-1)^{n_2}\,2^{n-n_2})
    \text{ and }
    \tfrac 12(\det(\lat) - 
    (-1)^{n_2}\,2^{n-n_2})
    ,
  \end{equation*}
  respectively, where $n_2$ is the integer described in
  Lemma~\ref{lem:weak-Jordan-decomposition}.
\end{Theorem}

\section{Examples}
\label{sec:examples}

In this section we study the $M_*(\ve^*)$-modules
$J_{\mathrm{parity},\lat}(\ve^*)$ for various lattices $\lat$, and, in
particular, give explicit formulas for their generators.

\subsection{Unimodular lattices}

\begin{Theorem}
  \label{eq:modules-over-unimodular-lattices}
  For any integral unimodular lattice $\lat$ of rank $n$, one has
  \begin{equation*}
    J_{\mathrm{parity},\lat}(\ve^*)
    =
    \begin{cases}
      M_*(\ve^*)\,\distjac {\lat}&\text{if $\mathrm{parity}$ is the parity of $n$}\\
      0&\text{otherwise},
    \end{cases}
  \end{equation*}
  where $\distjac {\lat}$ equals the function
  $\vartheta_{\lat,\lambda}$ defined
  in~\eqref{eq:basic-theta-functions} with some non-zero~$\lambda$ of
  the (one-dimensional) space $W(\lat)$.  The function $\distjac
  {\lat}$ is the (up to multiplication by a constant) unique non-zero
  Jacobi form of singular weight for $\lat$. It affords the character
  $\ve^{3n}$.
\end{Theorem}

\begin{proof}
  Since $\det(\lat)=1$ the space $\Theta(\lat)$ is one-dimensional,
  spanned by a Jacobi form $\distjac {\lat}$, which is then the only
  singular Jacobi form for the lattice $\lat$. Since $\distjac {\lat}
  (\tau+1,z) = e\big(\beta(r)\big) \distjac {\lat}(\tau,z)$, where $r$
  is any element in the shadow of $\lat$, and since by
  Lemma~\ref{lem:sigma-constant} $e\big(\beta(r)\big)=e_8(n)$, the
  character of $\distjac {\lat}$ equals $\ve^{3n}$. If~$\phi$ is an
  element of $J_{k,\lat}(\ve^h)$ then its theta expansion is of the
  form $\phi = f \distjac {\lat}$, where $f$ is in
  $M_{k-n/2}(\ve^{h-3n})$ (the space of modular forms on $\SL$
  transforming with the given character). But this space equals
  $\eta^lM_{k-n/2-l/2}\big(\SL\big)$, where $l\equiv h-3n\bmod 24$,
  $0\le l<24$. If $f\not=0$, then $h-n/2-l/2$ is even,
  i.e.~$k-h/2\equiv n\bmod 2$. This proves the theorem.
\end{proof}

The only unimodular (positive definite) lattices of rank $1\le n <8$
are the lattices~$\lat [\Z]^n$. For ranks $8\le n < 12$, the only
unimodular lattices are $\lat [E]_8\lsum\lat [\Z]^{n-8}$ and~$\lat
[\Z]^n$. For rank $n=12$ one has exactly three unimodular lattices,
namely $\lat [E]_8\lsum\lat [\Z]^{4}$ and $\lat [\Z]^{12}$ and a
lattice called $\lat [D]_{12}^+$. (See
\cite[Ch.16,~\S.4]{Conway-Sloane-3rd} for an account of the
classification of unimodular lattices.)  However, from the point of
Jacobi forms all unimodular lattices are somehow equivalent. More
precisly, all unimodular lattices fall into one stably equivalence
class and we then have natural isomorphism between the spaces of
Jacobi forms associated to different unimodular lattices as explained
in Theorem~\ref{thm:stably-isomorphic-indices}.  Still, Jacobi forms
belonging to different unimodular lattices might have a quite
different shape.  For example, in the case of the standard lattice
$\lat [\Z]^n$, the singular form $\vartheta_{\lat [\Z]^n,\lambda}$
equals the form in~\eqref{eq:singular-form-for-Zn} (where we
anticipated the notation $\distjac {\lat [\Z]^n}$), which has thus a
nice product expansion and thereby provides an eplicit decsription of
its divisor.  On the other hand, as a model for $\lat [E]_8$ one can
take the overlattice of $\lat [D]_8 = \ev{(\lat [\Z]^8)}$ consisting
of all vectors $y$ whose components are either all integers or else
are all in $\frac12+\Z$ and whose sum is an even integer. Using this
model we then have
\begin{equation}
  \label{eq:distjac-E8}
  \distjac {\lat [E]_8} (\tau,z)
  =
  \sum_{\begin{subarray}{c}
      r\in \Z^8\cup {\bf\frac 12}+\Z^8\\
      r_1+\cdots+r_8\in 2\Z
    \end{subarray}} q^{\frac {r^2}2}\,e\big( r\cdot z\big) ,
\end{equation}
(where ${\bf\frac12}=(\frac 12,\frac12,\dots,\frac12)$).

\subsection{The root lattice \texorpdfstring{$\lat [A]_2$}{A2}}

As a second example we consider an integral (positive definite)
lattice of rank~$2$ and determinant~$3$. Every such lattice is
isomorphic to the root lattice $\lat [A]_2=(A_2,\bif)$, which we can
realize as the set of Eisenstein integers~$A_2$ (algebraic integers in
$\Q(\sqrt 3)$) together with $\bif(x,y)=\tr (x\overline y)$ as
bilinear form, or as the lattice
\begin{equation*}
  \lat [A]_2
  =
  \big(\Z^2,(x,y)\mapsto x^t F y\big),
  \quad
  F = \mat 2112
  .
\end{equation*}
The lattice $\lat [A]_2$ has determinant~$3$ and accordingly the
$G$-module $\Theta(\lat [A]_2)$ is three-dimensional. It decomposes
into the direct sum of the subspaces $\Theta^{\mathrm{odd}}(\lat [A]_2)$
and $\Theta^{\mathrm{even}}(\lat [A]_2)$ of functions which are odd and
even in~$z$, respectively.  As follows from the transformation
formulas these subspaces are in fact $G$-invariant. The first one is
$1$-dimensional, spanned by the function $\distjac {\lat [A]_2} :=
\vartheta_{\lat [A]_2,\lambda_1}$, where $\lambda_1$ is an odd
function on the dual~$\dual {A_2}$ (the inverse different in $\Q(\sqrt
3)$ or $\Z^2F^{-1}$) which is constant on the $A_2$-cosets. If we
identify $\dual A_2/A_2$ with $\Z/3\Z$, an odd map is given by the
Legendre symbol composed with the canonical projection. We denote this
map by $\leg r3$. Accordingly we have
\begin{equation*}
  \distjac {\lat [A]_2} (\tau,z)
  =
  \sum_{r\in \dual A_2}
  \leg {r}3\,
  q^{\bif(r)}
  e\big(\bif(r,z))
  .
\end{equation*}
Note that for $r$ in $\dual A_2$, one has $\bif(r)\equiv 0,1/3\bmod
\Z$, from which it follows that $\distjac {\lat [A]_2}$ defines an
element in $J_{1,\lat [A]_2}(\ve^8)$.

For constructing functions which are even in $z$ note that
$\Theta^{\mathrm{even}}(\lat [A]_2)$ is two-dimensional. Hence the
$G$-submodule $\bigwedge^2\Theta^{\mathrm{even}}(\lat [A]_2)$ of
$\Theta^{\mathrm{even}}(\lat [A]_2)\otimes \Theta^{\mathrm{even}}(\lat
[A]_2)$ is one-dimensional, and provides thus a singular Jacobi form
$\distjac {\lat [A]_2\lsum\lat [A]_2,8}$ for the lattice $\lat
[A]_2\lsum \lat [A]_2$. Here we identify $\Theta(\lat [A]_2)\otimes
\Theta(\lat [A]_2)$ with functions on $\HP\times (\C\otimes A_2)\times
(\C\otimes A_2)$ in the obvious way (for this to work we also need
that any family of linearly independent functions in $\Theta(\lat)$
remains linearly independent when considered, for any given fixed
$\tau$ in~$\HP$, as functions in the lattice variable).  Using the
natural isomorphism $W(\lat [A]_2)\isom \Theta(\lat [A]_2)$, the form
$\distjac {\lat [A]_2\lsum\lat [A]_2,8}$ can be described as
$\vartheta_{\lat [A]_2,A_2\wedge A_2^{\mathrm{c}}}$, where we identify a
subset of $\dual A_2$ with its characteristic function and use
$A_2^{\mathrm{c}}$ for the complement of $A_2$ in $\dual A_2$.  We
therefore find
\begin{multline*}
  \distjac {\lat [A]_2\lsum\lat [A]_2,8} (\tau, z_1,z_2) =
  \\
  \sum_{s,r\in \dual A_2} \big( A_2(s)A_2^{\mathrm{c}}(r)-
  A_2^{\mathrm{c}}(s)A_2(r)\big)\, q^{\bif(s)+\bif(r)}\, e\big(
  \bif(s,z_1)+\bif(r,z_2) \big) .
\end{multline*}
Again it is easily checked that $\distjac {\lat [A]_2\lsum\lat
  [A]_2,8}$ becomes multiplied by $e_3(1)$ if one replaces~$\tau$ by
$\tau+1$, so that this function defines an element in~$J_{2,\lat
  [A]_2\lsum\lat[A]_2}(\ve^8)$.

From the singular Jacobi form $\distjac {\lat [A]_2\lsum\lat [A]_2,8}$
we derive the following Jacobi forms for the index $\lat [A]_2$:
\begin{equation*}
  \begin{split}
    \Eis {2,\lat [A]_2,8} (\tau,z) &= \sum_{s,r\in \dual A_2} \big(
    A_2(s)A_2^{\mathrm{c}}(r)- A_2^{\mathrm{c}}(s)A_2(r)\big)\,
    q^{\bif(s)+\bif(r)}\, e\big(\bif(r,z)\big)
    ,\\
    \Eis {4,\lat [A]_2,8} (\tau,z) 
    &= - E_2(\tau)\Eis {2,\lat [A]_2,8}(\tau,z)\\
    + 12 &\sum_{s,r\in \dual A_2} \big(
    A_2(s)A_2^{\mathrm{c}}(r)- A_2^{\mathrm{c}}(s)A_2(r)\big)\bif(s)\,
    q^{\bif(s)+\bif(r)}\,
    e\big(\bif(r,z)\big),\\
    \Eis {4,\lat [A]_2,0} &= (E_6\,\Eis {2,\lat [A]_2,8}+E_4\,\Eis
    {4,\lat [A]_2,8})/\eta^8
    ,\\
    \Eis {6,\lat [A]_2,0} &= (E_4^2\,\Eis {2,\lat [A]_2,8}+E_6\,\Eis
    {4,\lat [A]_2,8})/\eta^8 .
  \end{split}
\end{equation*}
If $\alpha$ denotes the embedding $s\mapsto (s,0)$ of $\lat [A]_2$
into $\lat [A]_2\lsum \lat [A]_2$, then $\Eis {2,\lat
  [A]_2,8}=\alpha^*\distjac {\lat [A]_2\lsum\lat [A]_2,8}$, from which
it follows that $\Eis {2,\lat [A]_2,8}\in J_{2,\lat [A]_2}(\ve^8)$ and
then, by a standard argument\footnote{If $\phi$ is a Jacobi form in
  $J_{k,\lat}(\ve^h)$ with theta expansion $\phi = \sum_{\lambda}
  h_\lambda \vartheta_{\lat,\lambda}$, then $\delta \phi :=
  \sum_{\lambda} \big(q\frac d{dq}h_\lambda\big)
  \vartheta_{\lat,\lambda}-\frac1{12}(k-\frac n2) E_2\phi$ defines a
  Jacobi form in $J_{k+2,\lat}(\ve^h)$, as follows easily using
  $E_2(A\tau)(c\tau+d)^{-2}=E_2(\tau)+\frac 6{\pi i}\frac
  c{c\tau+d}$.}, that $\Eis {4,\lat [A]_2,8}\in J_{4,\lat
  [A]_2}(\ve^8)$.  Clearly, $\Eis {4,\lat [A]_2,0} \in J_{4,\lat
  [A]_2}^!(1)$ and $\Eis {6,\lat [A]_2,0} \in J_{6,\lat
  [A]_2}^!(1)$. For proving that the latter two functions are also
holomorphic at infinity it suffices to note that they are quotients of
Jacobi cusp forms in which the exponents of the $q$-powers run through
$\frac 13+\Z$, divided by $\eta^8$.

Using these Jacobi forms on $\lat [A_2]$, we can now prove

\begin{Theorem}
  \label{thm:modules-over-A2}
  For the root lattice $\lat [A]_2$, one has (using
  $S:=M_*(\ve^*)=\C[E_4,E_6,\eta]$)
  \begin{equation*}
    \begin{split}
      J_{\mathrm{odd},\lat [A]_2}(\ve^*) &=
      S\,\distjac {\lat [A_2]},\\
      J_{\mathrm{even},\lat [A]_2}(\ve^*) &= S\,\Eis {2,\lat [A]_2,8} +S\,\Eis
      {4,\lat [A]_2,8} +S\,\Eis {4,\lat [A]_2,0} +S\,\Eis {6,\lat
        [A]_2,0} .
    \end{split}
  \end{equation*}
  For the $S$-module $J_{\mathrm{even},\lat [A]_2}(\ve^*)$ one has the
  following exact sequence
  \begin{equation*}
    S^4
    \xrightarrow{
      \times\left[\begin{smallmatrix}
          E_6&E_4&-\eta^8&0\\ E_4^2&E_6&0&-\eta^8\\ 12^3\eta^{16}&0&E_6&-E_4\\0&12^3\eta^{16}&-E_4^2&E_6
        \end{smallmatrix}\right]}
    S^4
    \xrightarrow{\times\left[\begin{smallmatrix}\Eis {2,\lat [A]_2,8}\\ \Eis {4,\lat [A]_2,8}\\ \Eis {4,\lat [A]_2,0}\\ \Eis {6,\lat [A]_2,0}\end{smallmatrix}\right]}
    J_{\mathrm{even},\lat [A]_2}(\ve^*)
    \xrightarrow{}
    0
    .
  \end{equation*}
  Here $\times M$ is the map $x=(x_1,\dots,x_r)\mapsto x_1m_1+\cdots
  x_nm_r$, where~$m_j$ denotes the $j$-th row of the matrix $M$.
\end{Theorem}
\begin{Remark}
  1.~The lattice $\lat [A]_2$ can also be realized as the submodule of
  vectors in~$\lat [\Z]^3$ equipped with the standard scalar product
  of $\lat [\Z]^3$. In other words we have an embedding
  $\alpha=(\alpha_1,\alpha_2,\alpha_3)$ of $\lat [A]_2$ in to $\lat
  [\Z]^3$. The pullback $\alpha^*\distjac {\lat [\Z]^3}$ is odd, and
  hence an $S$-multiple of $\distjac {\lat [A]_2}$. It follows that
  \begin{equation*}
    \distjac {\lat [A]_2}(\tau,z) =
    \eta^{-1}(\tau)\,
    \vartheta(\tau,\alpha_1(z))\,\vartheta(\tau,\alpha_2(z))\,\vartheta(\tau,\alpha_1(z)+\alpha_2(z))
    .
  \end{equation*}

  2.~If we view $J_{\mathrm{even},\lat [A]_2}(\ve^*)$ as a graded module over
  the polynomial ring $R$ in three variables $x$, $y$, $z$ with
  weights $4$, $6$ and $1/2$, respectively, via the action
  $(f,\phi)\mapsto f(E_4.E_6,\eta)\phi$, then the module
  $S_1\big(J_{\mathrm{even},\lat [A]_2}(\ve^*)\big)$ of first syzygies is still
  spanned by the rows of the $4\times 4$-matrix $B$ occurring in the
  exact sequence of the theorem (which has, of course, to be viewed as
  a matrix over $R$). Indeed, if $v\in R^4$ is a relation between the
  given generators of $J_{\mathrm{even},\lat [A]_2}(\ve^*)$, then, by the
  theorem, $v$ is an element of $R^4B+K$, where $K$ is the kernel of
  the evaluation map $w\mapsto w(E_4,E_6,\eta)$ from $R^4$ to
  $S^4$. But this kernel equals $\delta R^4$, where $\delta=(12^3
  z^{24} - x^3+y^2)$, and a simple calculation shows that the
  determinant of $B$ equals $\delta^2$ and that $\delta$ divides the
  entries of the adjoint matrix of $B$, which implies that $K$ is
  contained in $R^4B$. From $\det(B)=\delta^2\not=0$ we also see that
  $S_1\big(J_{\mathrm{even},\lat [A]_2}(\ve^*)\big)$ is free.
\end{Remark}
\begin{proof}[Proof of Theorem~\ref{thm:modules-over-A2}]
  The first statement is nothing else but a restatement of the theta
  expansion for Jacobi forms on $\lat [A]_2$ which are odd in the
  lattice variable on using that $\distjac {\lat [A]_2}$ is the only
  funtion in $\Theta(\lat [A]_2)$ which is odd in the lattice variable
  as we saw above.

  For the second statement we use the dimension formula to find the
  following table, where the entry in the $h$-th row and $k$-th column
  is the dimension of the subspace of forms in $J_{k,\lat
    [A]_2}(\ve^h)$ which are even in the lattice variable. From the
  discussion preceding the theorem we know that there is no such form
  in singular weight (which we therefore omitted from the table).
  \begin{equation*}
    \begin{smallmatrix}
      &\frac32&2&\frac52&3&\frac72&4&\frac92 &5&\frac{11}2&6&\frac{13}2&7&\frac{15}2&8&\frac{17}2&9&\frac{19}2&10&\frac{21}2&11&\frac{23}2\\
      0 & & & & & &1 & & & &1 & & & &1 & & & &2 & & & \\
      1 & & & & & & &1 & & & &1 & & & &1 & & & &2 & & \\
      2 & & & & & & & &1 & & & &1 & & & &1 & & & &2 & \\
      3 & & & & & & & & &1 & & & &1 & & & &1 & & & &2 \\
      4 & & & & & & & & & &1 & & & &1 & & & &1 & & & \\
      5 & & & & & & & & & & &1 & & & &1 & & & &1 & & \\
      6 & & & & & & & & & & & &1 & & & &1 & & & &1 & \\
      7 & & & & & & & & & & & & &1 & & & &1 & & & &1 \\
      8 & &1 & & & &1 & & & &1 & & & &2 & & & &2 & & & \\
      9 & & &1 & & & &1 & & & &1 & & & &2 & & & &2 & & \\
      10 & & & &1 & & & &1 & & & &1 & & & &2 & & & &2 & \\
      11 & & & & &1 & & & &1 & & & &1 & & & &2 & & & &2 \\
      12 & & & & & &1 & & & &1 & & & &1 & & & &2 & & & \\
      13 & & & & & & &1 & & & &1 & & & &1 & & & &2 & & \\
      14 & & & & & & & &1 & & & &1 & & & &1 & & & &2 & \\
      15 & & & & & & & & &1 & & & &1 & & & &1 & & & &2 \\
      16 & & & & & & & & & &1 & & & &1 & & & &1 & & & \\
      17 & & & & & & & & & & &1 & & & &1 & & & &1 & & \\
      18 & & & & & & & & & & & &1 & & & &1 & & & &1 & \\
      19 & & & & & & & & & & & & &1 & & & &1 & & & &1 \\
      20 & & & & & & & & & & & & & &1 & & & &1 & & & \\
      21 & & & & & & & & & & & & & & &1 & & & &1 & & \\
      22 & & & & & & & & & & & & & & & &1 & & & &1 & \\
      23 & & & & & & & & & & & & & & & & &1 & & & &1 \\
    \end{smallmatrix}
  \end{equation*}
  (Strictly speaking the dimension formula gives only the dimensions
  for weights $k\le 3$. Hence, for each $h$ there is exactly one
  weight $1\le k<3$ in $\frac h2 +2\Z$ which is not known. The
  validity of the entries of the table for these weights and
  $h\not=8,9$ can be verified in each case using either of the
  following four rules: (i)~$J_{k,\lat [A]_2}(\ve^h)=0$ if
  $J_{k+2,\lat [A]_2}(\ve^{h+2})=0$, (ii)~$J_{k,\lat [A]_2}(\ve^h)=0$
  if $k+1\ge3$ and $J_{k+1,\lat [A]_2}(\ve^{h+1})=0$, (iii)~$J_{k,\lat
    [A]_2}(\ve^h)=0$ if $J_{k+4,\lat [A]_2}(\ve^{h})=0$. These rules
  follow from the fact that multiplication by $\eta^4$, $\eta^2$ and
  $E_4$, respectively, take the first space into the second). Rule (i)
  applies to all $h$ except $h=4,5,7,8,9,20,21$, rule~(ii) to
  $h=4,5,20,21$, rule~(iii) to $h=7$. Finally, for $h=8,9$ we note
  that~$J_{2,\lat [A]_2}(\ve^8)$ and $J_{5/2,\lat [A]_2}(\ve^9)$ are
  at least one-dimensional since they contain $\Eis {2,\lat [A]_2,8}$
  and $\eta \Eis {2,\lat [A]_2,8}$, respectively, and that they are
  not of higher dimension as multiplication of these spaces by
  $\eta^4$ and inspecting the dimension of the space containing the
  image shows.)

  From Theorem~\ref{thm:Jacobi-forms-as-module-over-CE4E6} we know
  that, for each $h$, the $M_*(1)$-module $J_{\mathrm{even},\lat
    [A]_2}(\ve^h)$ is free of rank~$2$. From this and the table of
  dimensions we see that $J_{\mathrm{even},\lat [A]_2}(\ve^h)$, for
  $0\le h\le 7$ is generated by $\eta^h\Eis {4,\lat [A]_2,0}$
  and~$\eta^h\Eis {6,\lat [A]_2,0}$, and, for $8\le h\le 23$ is
  generated by $\eta^h\Eis {2,\lat [A]_2,8}$ and $\eta^h\Eis {4,\lat
    [A]_2,8}$. This proves the second statement.

  Finally, for proving the third statement, consider a relation
  $f=(f_1,\dots,f_4)$ in $S^4$ between the given generators of
  $J_{\mathrm{even},\lat [A]_2}(\ve^*)$. We have to show that $f$ is a
  $S$-linear combination of the rows of the $4\times 4$ matrix~$B$
  occurring in the exact sequence. (That $S^4B$ is contained in the
  $S$-module of relations is obvious.)  For this we can assume that
  the $f_i$ are graded elements of $S$ (i.e.~that each $f_i$ is in a
  space $\eta^hM_k(1)$ for a suitable $k$ and $0\le h < 24$. Replacing
  in the identity $f_1\Eis {2,\lat [A]_2,8} + f_2\Eis {4,\lat [A]_2,8}
  +f_3\Eis {4,\lat [A]_2,0} +f_4\Eis {6,\lat [A]_2,0}=0$ the Jacobi
  forms $\Eis {4,\lat [A]_2,0}$ and $\Eis {6,\lat [A]_2,0}$ by their
  expressions in terms of $\Eis {2,\lat [A]_2,8}$ and $\Eis {4,\lat
    [A]_2,8}$ and using that $\Eis {2,\lat [A]_2,8}$ and $\Eis {4,\lat
    [A]_2,8}$ are linearly independent over $M_*(1)$ yields relations
  among the $f_j$, which imply that~$v$ is indeed in $S^4B$. We leave
  the details to the reader.
\end{proof}

\subsection{The root lattice \texorpdfstring{$\lat [D]_3$}{D3}}

The root lattice $\lat [D]_3$ can be realized as the sublattice of
$\lat [\Z]^3$ consisting of those vectors whose sum of components is
even. In other words, $\lat [D]_3=\ev{(\lat [\Z]^3)}$. Note also that
$\lat [D]_3$ is isomorphic to the lattice $\lat [A]_3$, which can be
realized as the $\Z$-module of integral vectors in $\Z^4$ whose
entries sum up to $0$, equipped with the standard scalar product of
$\lat [\Z]^4$. There exist modulo left and right multiplication by
automorphisms exactly two proper embeddings of $\lat [D]_3$ into the
standard lattices $\lat [\Z]^N$, namely, the inclusion $\iota_1$ for
$N=3$, and the other being the identification $\iota_2$ with the
mentioned model of $\lat [A]_3$ in $\lat [\Z]^4$. Accordingly we
obtain the two Jacobi forms $\iota_1^*\distjac {\lat [\Z]^3}$ and
$\iota_2^*\distjac {\lat [\Z]^4}$. The first form is the only singular
weight form for $\lat [D]_3$, the second one the only form of weight
$2$ and character $12$, as it can be easily deduced using the
dimension formula.

If we write $\iota_1=(z_1,z_2,z_3)$, then for $\iota_2$ we can take
\begin{equation*}
  \iota_2=\tfrac12(z_1+z_2-z_3,z_1-z_2+z_3,-z_1+z_2+z_3,-z_1-z_2-z_3,)
  .
\end{equation*}
For the corresponding pullback we then obtain
\begin{equation*}
  \iota_2^*\distjac {\lat [\Z]^4}(\tau,\_)
  =
  -
  \vartheta(\tau,\tfrac{z_1+z_2-z_3}2)\,
  \vartheta(\tau,\tfrac{z_1-z_2+z_3}2)\,
  \vartheta(\tau,\tfrac{-z_1+z_2+z_3}2)\,
  \vartheta(\tau,\tfrac{z_1+z_2+z_3}2)
  ,
\end{equation*}
where we also write $z_j$ for the $\C$-linear continuation of $z_j$
(and where we used that~$\vartheta(\tau,z)$ is odd in $z$).
 
From the supplement to Theorem~\ref{thm:dimension-formula} we see
that, for given $k\ge 4$, the space $J_{k, \lat [D]_3}(1)$ contains
exactly one or no Eisenstein series accordingly as $k$ is even or
odd. Moreover, from the dimension formula we know that $J_{k, \lat
  [D]_3}(1)$ is one dimensional for $k=4,6$, hence spanned by its
Eisenstein series. For obtaining formulas for the Eisenstein series in
these weights, we consider embeddings of $\lat [D]_3$ into $\lat
[E]_8$. If we take for $\lat [E]_8$ the model discussed in the last
but not least subsection and for $\lat [D]_3$ the lattice $\ev{(\lat
  [\Z]^3)}$, such an embedding is given by $\alpha:(x,y,z)\mapsto
(0,0,0,0,x,y,z,0)$. In this way we obtain the forms
\begin{equation*}
  \Eis {4,\lat [D]_3,0} := \alpha^* \distjac{\lat [E]_8},
  \qquad
  \Eis {6,\lat [D]_3,0} := \delta \alpha^* \distjac{\lat [E]_8}   
  ,
\end{equation*}
where we use again the differential operator $\delta$ introduced in
the preceding discussion of $\lat [A]_2$ and $\distjac {\lat [E]_8}$

from~\eqref{eq:distjac-E8}.

For the Poincar\'e-Hilbert series of $J_{\mathrm{odd},\lat [D]_3}(1)$ and
$J_{\mathrm{even},\lat [D]_3}(1)$ we find
\begin{equation*}
  \begin{split}
    P_{\mathrm{odd},\lat [D]_3.0}&=\tfrac {t^9}{(1-t^4)(1-t^6)}\\
    P_{\mathrm{even},\lat [D]_3.0}&= t^{4} + t^{6} + 2t^{8} + 2t^{10} +
    3t^{12} + \cdots =\tfrac {t^4+t^6+t^8}{(1-t^4)(1-t^6)} .
  \end{split}
\end{equation*}
(The dimension formula gives us only explicit information for weights
$k\ge 7/2$.  However, $J_{2,\lat [D]_3}(1)=0$ and $J_{3,\lat
  [D]_3}(1)=0$ since $\eta^4J_{2,\lat [D]_3}(1) \subseteq J_{4,\lat
  [D]_3}(\ve^4)$ and since $E_4J_{3,\lat [D]_3}(1) \subseteq J_{7,\lat
  [D]_3}(1)$, but $J_{4,\lat [D]_3}(\ve^4)=0$ and $J_{7,\lat
  [D]_3}(1)=0$ by the dimension formula). We note that
$\eta^{12}\,\iota_2^*\distjac {\lat [\Z]^4}$ is a cusp form in
$J_{8,\lat [D]_3}(1)$. In fact, it is the only one in this space, and
the one of smallest weight and trivial character on $\lat [D]_3$.

Summarizing, we have proved
\begin{Theorem}
  \label{thm:modules-over-D3}
  For the root lattice $\lat [D]_3$, one has
  \begin{equation*}
    \begin{split}
      J_{\mathrm{odd},\lat [D]_3}(1) &=
      M_*(1)\,\eta^{15}\,\iota_1^*\distjac {\lat [\Z]^3},\\
      J_{\mathrm{even},\lat [D]_3}(1) &= M_*(1)\,\Eis {4,\lat [D]_3,0} \oplus
      M_*(1)\,\Eis {6,\lat [D]_3,0} \oplus
      M_*(1)\,\eta^{12}\,\iota_2^*\distjac {\lat [\Z]^4} .
    \end{split}
  \end{equation*}
\end{Theorem}

Using the dimension formula and following the method of proof of
Theorem~\ref{thm:modules-over-A2} it can be verified that in fact
$J_{\mathrm{odd},\lat [D]_3}(\ve^*)$ is generated as module over
$M_*(\ve^*)=\C[E_4,E_6,\eta]$ by $\iota_1^*\distjac {\lat [\Z]^3}$,
and that $J_{\mathrm{odd},\lat [D]_3}(\ve^*)$ as module over $M_*(\ve^*)$ needs
$9$ generators, which are of weight $4,6,8$ and character $1$, weight
$4+\tfrac12,6+\tfrac12,8+\tfrac12$ and character $\ve^9$, and weight
$2,4,6$ and character $\ve^{12}$. As we saw at the end of
Section~\ref{sec:structure-theorems} these generators can be
constructed as linear combination of the three given generators for
$J_{\mathrm{even},\lat [D]_3}(1)$ with coefficients in
$\C[E_4,E_6,\eta,\eta^{-1}]$.

\bibliography{joli} \bibliographystyle{alpha}

\begin{thebibliography}{Woh64}

\bibitem[Ara92]{Arakawa}
Tsuneo Arakawa.
\newblock Selberg zeta functions and {J}acobi forms.
\newblock In {\em Zeta functions in geometry ({T}okyo, 1990)}, volume~21 of
  {\em Adv. Stud. Pure Math.}, pages 181--218. Kinokuniya, Tokyo, 1992.

\bibitem[Boy15]{Boylan-Thesis}
Hatice Boylan.
\newblock {\em Jacobi forms, finite quadratic modules and {W}eil
  representations over number fields}, volume 2130 of {\em Lecture Notes in
  Mathematics}.
\newblock Springer, Cham, 2015.
\newblock With a foreword by Nils-Peter Skoruppa.

\bibitem[CG13]{Clery-Gritsenko}
F.~Cl\'{e}ry and V.~Gritsenko.
\newblock Modular forms of orthogonal type and {J}acobi theta-series.
\newblock {\em Abh. Math. Semin. Univ. Hambg.}, 83(2):187--217, 2013.

\bibitem[CS89]{Conway-Sloane-V}
J.~H. Conway and N.~J.~A. Sloane.
\newblock Low-dimensional lattices. {V}. {I}ntegral coordinates for integral
  lattices.
\newblock {\em Proc. Roy. Soc. London Ser. A}, 426(1871):211--232, 1989.

\bibitem[CSB99]{Conway-Sloane-3rd}
J.~Conway, N.J.A. Sloane, and E.E. BANNAI.
\newblock {\em Sphere Packings, Lattices and Groups}.
\newblock Die Grundlehren der mathematischen Wissenschaften in
  Einzeldarstellungen. Springer, 1999.

\bibitem[Ded77]{Dedekind1}
Richard Dedekind.
\newblock Schreiben an {H}errn {B}orchardt \"uber die {T}heorie der
  elliptischen {M}odul-{F}unctionen.
\newblock {\em J. Reine Angew. Math.}, 83:265--292, 1877.

\bibitem[Ded30]{Dedekind2}
Richard Dedekind.
\newblock Erl\"auterungen zu zwei {F}ragmenten von {R}iemann.
\newblock In Robert Fricke, Emmy Noether, and \"Oystein Ore, editors, {\em
  Richard Dedekind, Gesammelte mathematische Werke}, volume~1, pages 159--173.
  Friedr. Vieweg \& Sohn Akt.-Ges., Braunschweig, 1930.

\bibitem[ES95]{Eholzer-Skoruppa}
Wolfgang Eholzer and Nils-Peter Skoruppa.
\newblock Modular invariance and uniqueness of conformal characters.
\newblock {\em Comm. Math. Phys.}, 174(1):117--136, 1995.

\bibitem[EZ85]{Eichler-Zagier}
Martin Eichler and Don Zagier.
\newblock {\em The theory of {J}acobi forms}, volume~55 of {\em Progress in
  Mathematics}.
\newblock Birkh\"auser Boston Inc., Boston, MA, 1985.

\bibitem[HBJ92]{Skoruppa:Manifolds-and-Modular-Forms}
Friedrich Hirzebruch, Thomas Berger, and Rainer Jung.
\newblock {\em Manifolds and modular forms}.
\newblock Aspects of Mathematics, E20. Friedr. Vieweg \& Sohn, Braunschweig,
  1992.
\newblock With appendices by Nils-Peter Skoruppa and by Paul Baum.

\bibitem[Hec44]{Hecke-1}
E.~Hecke.
\newblock Herleitung des {E}uler-{P}roduktes der {Z}etafunktion und einiger
  {$L$}-{R}eihen aus ihrer {F}unktionalgleichung.
\newblock {\em Math. Ann.}, 119:266--287, 1944.

\bibitem[Jac29]{Jacobi}
C.~G.~J. Jacobi.
\newblock {\em Fundamenta Nova Theoriae Functionum Ellipticarum}.
\newblock Borntraeger, 1829.

\bibitem[Klo46]{Kloosterman-I}
H.~D. Kloosterman.
\newblock The behaviour of general theta functions under the modular group and
  the characters of binary modular congruence groups. {I}.
\newblock {\em Ann. of Math. (2)}, 47:317--375, 1946.

\bibitem[MH73]{Milnor-Husemoller}
John Milnor and Dale Husemoller.
\newblock {\em Symmetric bilinear forms}.
\newblock Springer-Verlag, New York, 1973.
\newblock Ergebnisse der Mathematik und ihrer Grenzgebiete, Band 73.

\bibitem[Ran77]{Rankin}
Robert~A. Rankin.
\newblock {\em Modular forms and functions}.
\newblock Cambridge University Press, Cambridge, 1977.

\bibitem[RS98]{Sloane-Rains}
E.~M. Rains and N.~J.~A. Sloane.
\newblock The shadow theory of modular and unimodular lattices.
\newblock {\em J. Number Theory}, 73(2):359--389, 1998.

\bibitem[Shi73]{Shimura}
Goro Shimura.
\newblock On modular forms of half integral weight.
\newblock {\em Ann. of Math. (2)}, 97:440--481, 1973.

\bibitem[Sko85]{Skoruppa:Dissertation}
Nils-Peter Skoruppa.
\newblock {\em \"{U}ber den {Z}usammenhang zwischen {J}acobiformen und
  {M}odulformen halbganzen {G}ewichts}.
\newblock Bonner Mathematische Schriften [Bonn Mathematical Publications], 159.
  Universit\"at Bonn Mathematisches Institut, Bonn, 1985.
\newblock Dissertation, Rheinische Friedrich-Wilhelms-Universit{\"a}t, Bonn,
  1984.

\bibitem[Sko07]{Skoruppa:Chern-Classes}
Nils-Peter Skoruppa.
\newblock Jacobi forms of degree one and {W}eil representations.
\newblock In Tomoyoshi Ibukiyama, editor, {\em {S}iegel Modular Forms and
  Abelian Varieties, Proceedings of the 4-th Spring Conference on Modular Forms
  and Related Topics}, pages 216--229, Osaka, 2007. Ryushido.

\bibitem[Sko08]{Skoruppa-Schiermonnikoog}
Nils-Peter Skoruppa.
\newblock Jacobi forms of critical weight and {W}eil representations.
\newblock In {\em Modular forms on {S}chiermonnikoog}, pages 239--266.
  Cambridge Univ. Press, Cambridge, 2008.

\bibitem[Sko23]{Skoruppa-fqm}
Nils-Peter Skoruppa.
\newblock Finite quadratic modules, {W}eil representations and vector valued
  modular forms.
\newblock in preparation, 2010--2023.

\bibitem[SS77]{Serre-Stark}
J.-P. Serre and H.~M. Stark.
\newblock Modular forms of weight {$1/2$}.
\newblock In {\em Modular functions of one variable, {VI} ({P}roc. {S}econd
  {I}nternat. {C}onf., {U}niv. {B}onn, {B}onn, 1976)}, pages 27--67. Lecture
  Notes in Math., Vol. 627. Springer, Berlin, 1977.

\bibitem[Str13]{Stroemberg}
Fredrik Str\"{o}mberg.
\newblock Weil representations associated with finite quadratic modules.
\newblock {\em Math. Z.}, 275(1-2):509--527, 2013.

\bibitem[Wal72]{Wall-II}
C.~T.~C. Wall.
\newblock Quadratic forms on finite groups. {II}.
\newblock {\em Bull. London Math. Soc.}, 4:156--160, 1972.

\bibitem[Woh64]{Wohlfahrt}
Klaus Wohlfahrt.
\newblock An extension of {F}. {K}lein's level concept.
\newblock {\em Illinois J. Math.}, 8:529--535, 1964.

\end{thebibliography}

\end{document}